\documentclass[3p,11pt]{elsarticle}
\usepackage{amssymb}
\usepackage{amsmath}
\usepackage{amsthm}

\DeclareMathAlphabet{\mathcal}{OMS}{cmsy}{m}{n}

\makeatletter
\def\ps@pprintTitle{%
 \let\@oddhead\@empty
 \let\@evenhead\@empty
 \def\@oddfoot{\centerline{\thepage}}%
 \let\@evenfoot\@oddfoot}
\makeatother

\newcommand{\bbC}{\mathbb{C}}
\newcommand{\bbF}{\mathbb{F}}
\newcommand{\bbH}{\mathbb{H}}

\newcommand{\bbR}{\mathbb{R}}
\newcommand{\bbT}{\mathbb{T}}
\newcommand{\bbZ}{\mathbb{Z}}

\newcommand{\bfA}{\mathbf{A}}
\newcommand{\bfB}{\mathbf{B}}
\newcommand{\bfc}{\mathbf{c}}
\newcommand{\bfC}{\mathbf{C}}

\newcommand{\bfG}{\mathbf{G}}

\newcommand{\bfI}{\mathbf{I}}
\newcommand{\bfJ}{\mathbf{J}}
\newcommand{\bfP}{\mathbf{P}}

\newcommand{\bfx}{\mathbf{x}}

\newcommand{\bfy}{\mathbf{y}}

\newcommand{\bfz}{\mathbf{z}}

\newcommand{\bfone}{\boldsymbol{1}}
\newcommand{\bfzero}{\boldsymbol{0}}
\newcommand{\bfdelta}{\boldsymbol{\delta}}
\newcommand{\bfDelta}{\boldsymbol{\Delta}}

\newcommand{\bfGamma}{\boldsymbol{\Gamma}}
\newcommand{\bfchi}{\boldsymbol{\chi}}
\newcommand{\bfphi}{\boldsymbol{\varphi}}
\newcommand{\bfPhi}{\boldsymbol{\Phi}}
\newcommand{\bfpsi}{\boldsymbol{\psi}}
\newcommand{\bfPsi}{\boldsymbol{\Psi}}
\newcommand{\bfTheta}{\boldsymbol{\Theta}}

\newcommand{\calB}{\mathcal{B}}
\newcommand{\calD}{\mathcal{D}}
\newcommand{\calE}{\mathcal{E}}
\newcommand{\calG}{\mathcal{G}}

\newcommand{\calN}{\mathcal{N}}
\newcommand{\calR}{\mathcal{R}}

\newcommand{\calU}{\mathcal{U}}
\newcommand{\calV}{\mathcal{V}}

\newcommand{\calX}{\mathcal{X}}
\newcommand{\calY}{\mathcal{Y}}

\newcommand{\rmB}{\mathrm{B}}
\newcommand{\rmc}{\mathrm{c}}
\newcommand{\rmC}{\mathrm{C}}
\newcommand{\rmd}{\mathrm{d}}

\newcommand{\rmQ}{\mathrm{Q}}
\newcommand{\rmT}{\mathrm{T}}

\newcommand{\Tr}{\operatorname{Tr}}

\newcommand{\sgn}{{\operatorname{sgn}}}
\newcommand{\dist}{\operatorname{dist}}
\newcommand{\rank}{\operatorname{rank}}
\newcommand{\BIBD}{{\operatorname{BIBD}}}
\newcommand{\Span}{\operatorname{span}}

\newcommand{\ETF}{{\operatorname{ETF}}}
\newcommand{\ECTFF}{{\operatorname{ECTFF}}}
\newcommand{\EITFF}{{\operatorname{EITFF}}}

\newcommand{\Fro}{\mathrm{Fro}}
\newcommand{\op}{\mathrm{op}}

\newcommand{\abs}[1]{|{#1}|}

\newcommand{\Bigparen}[1]{\Bigl({#1}\Bigr)}
\newcommand{\biggparen}[1]{\biggl({#1}\biggr)}

\newcommand{\bigbracket}[1]{\bigl[{#1}\bigr]}
\newcommand{\Bigbracket}[1]{\Bigl[{#1}\Bigr]}

\newcommand{\set}[1]{\{{#1}\}}
\newcommand{\bigset}[1]{\bigl\{{#1}\bigr\}}

\newcommand{\norm}[1]{\|{#1}\|}

\newcommand{\ip}[2]{\langle{#1},{#2}\rangle}

\newtheorem{theorem}{Theorem}[section]
\newtheorem{lemma}[theorem]{Lemma}

\theoremstyle{definition}
\newtheorem{definition}[theorem]{Definition}
\newtheorem{example}[theorem]{Example}
\newtheorem{remark}[theorem]{Remark}

\begin{document}
\begin{frontmatter}
\title{Grassmannian codes from paired difference sets}

\author[AFIT]{Matthew Fickus}
\ead{Matthew.Fickus@gmail.com}
\author[IA]{Joseph W.\ Iverson}
\author[SDSU]{John Jasper}
\author[CSU]{Emily J.\ King}

\address[AFIT]{Department of Mathematics and Statistics, Air Force Institute of Technology, Wright-Patterson AFB, OH 45433}
\address[IA]{Department of Mathematics, Iowa State University, Ames, IA 50011}
\address[SDSU]{Department of Mathematics and Statistics, South Dakota State University, Brookings, SD 57007}
\address[CSU]{Department of Mathematics, Colorado State University, Fort Collins, CO 80523}

\begin{abstract}
An equiangular tight frame (ETF) is a sequence of vectors in a Hilbert space that achieves equality in the Welch bound and so has minimal coherence.
More generally,
an equichordal tight fusion frame (ECTFF) is a sequence of equi-dimensional subspaces of a Hilbert space that achieves equality in Conway, Hardin and Sloane's simplex bound.
Every ECTFF is a type of optimal Grassmannian code,
that is, an optimal packing of equi-dimensional subspaces of a Hilbert space.
We construct ECTFFs by exploiting new relationships between known ETFs.
Harmonic ETFs equate to difference sets for finite abelian groups.
We say that a difference set for such a group is ``paired" with a difference set for its Pontryagin dual when the corresponding subsequence of its harmonic ETF happens to be an ETF for its span.
We show that every such pair yields an ECTFF.
We moreover construct an infinite family of paired difference sets using quadratic forms over the field of two elements.
Together this yields two infinite families of real ECTFFs.
\end{abstract}

\end{frontmatter}

\section{Introduction}

The \textit{chordal distance} between two $R$-dimensional subspaces $\calU_1$ and $\calU_2$ of a $D$-dimensional real or complex Hilbert space $\bbH$ is
$\dist(\calU_1,\calU_2)
:=2^{-\frac12}\norm{\bfP_1-\bfP_2}_\Fro
=[R-\Tr(\bfP_1\bfP_2)]^{\frac12}$
where $\bfP_1$ and $\bfP_2$ are their respective rank-$R$ orthogonal projection operators.
Conway, Hardin and Sloane~\cite{ConwayHS96} showed that the minimum pairwise chordal distance between the members of any sequence $\set{\calU_n}_{n=1}^N$ of $R$-dimensional subspaces of $\bbH$ satisfies the \textit{simplex bound}:
\begin{equation}
\label{eq.simplex bound}
\smash{\min_{n_1\neq n_2}\dist(\calU_{n_1},\calU_{n_2})
\leq\bigbracket{\tfrac{R(D-R)}{D}\tfrac{N}{N-1}}^{\frac12}.}
\end{equation}
In modern parlance~\cite{KutyniokPCL09},
they further showed that such a sequence $\set{\calU_n}_{n=1}^N$ achieves equality in~\eqref{eq.simplex bound} if and only if it is an \textit{equichordal tight fusion frame} (ECTFF) for $\bbH$,
namely when $\dist(\calU_{n_1},\calU_{n_2})$ is constant over all $n_1\neq n_2$ (equichordality) and $\sum_{n=1}^N\bfP_n=A\bfI$ for some $A>0$ (tightness).
When such an ECTFF for $\bbH$ exists it is thus an optimal \textit{Grassmannian code}, that is,
an optimal packing (with respect to the chordal distance) of $N$ points on the \textit{Grassmannian} (space) that consists of all $R$-dimensional subspaces of the $D$-dimensional space $\bbH$.
When $R=1$ the simplex bound~\eqref{eq.simplex bound} reduces to the \textit{Welch bound}~\cite{Welch74,StrohmerH03} on the \textit{coherence} of $N$ nonzero vectors $\set{\bfphi_n}_{n=1}^{N}$ in $\bbH$:
\begin{equation}
\label{eq.Welch bound}
\max_{n_1\neq n_2}
\tfrac{\abs{\ip{\bfphi_{n_1}}{\bfphi_{n_2}}}}{\norm{\bfphi_{n_1}}\norm{\bfphi_{n_2}}}
\geq\bigbracket{\tfrac{N-D}{D(N-1)}}^{\frac12}.
\end{equation}
In this case, an ECTFF for $\bbH$ equates to an \textit{equiangular tight frame} (ETF) for $\bbH$, namely to a sequence~$\set{\bfphi_n}_{n=1}^{N}$ of nonzero equal-norm vectors in $\bbH$ that achieves equality in~\eqref{eq.Welch bound}.
More generally an ECTFF for $\bbH$ will have minimal \textit{block coherence} $\max_{n_1\neq n_2}\norm{\bfP_{n_1}\bfP_{n_2}}_\op$ if its subspaces are \textit{equi-isoclinic}~\cite{LemmensS73b},
that is, satisfy $\bfP_{n_1}\bfP_{n_2}\bfP_{n_1}=\sigma^2\bfP_{n_1}$ for some $\sigma\geq0$ and all $n_1\neq n_2$~\cite{DhillonHST08}.
Such an ECTFF is called an \textit{equi-isoclinic tight fusion frame} (EITFF) for $\bbH$.

ETFs, ECTFFs and EITFFs arise in various applications,
including compressed sensing~\cite{EldarKB10,BajwaCM12,BandeiraFMW13,CalderbankTX15},
quantum information theory~\cite{Zauner99,RenesBSC04},
wireless communication~\cite{StrohmerH03,Bodmann07},
and algebraic coding theory~\cite{JasperMF14}.
Much of the related literature is devoted to the \textit{existence problem}:
for what $D$, $N$ and $R$ does there exist an $\ECTFF(D,N,R)$, that is, an ECTFF for a $D$-dimensional Hilbert space that consists of $N$ subspaces of dimension $R$?
Moreover, in such cases, when can these subspaces be chosen to be equi-isoclinic and/or real?
Most positive existence results involve explicit construction from some type of combinatorial design.
See~\cite{FickusM16} for a survey of known $\ETF(D,N)$ (i.e., $\ECTFF(D,N,1)$).
Several constructions of $\ECTFF(D,N,R)$ with $R>1$ are known.
Some of these actually yield $\EITFF(D,N,R)$:
one can tensor an ETF with an orthonormal basis (ONB)~\cite{LemmensS73b,CalderbankTX15,King21},
or convert a complex and/or quaternionic ETF into an EITFF over a subfield~\cite{Hoggar77,EtTaoui20,Waldron20},
or exploit a complex \textit{conference} matrix~\cite{EtTaoui18,BlokhuisBE18}.
Other methods yield ECTFFs that are not necessarily equi-isoclinic,
including constructions from
\textit{quadratic residues}~\cite{CalderbankHRSS99,ZhangG18}
and their generalizations~\cite{KocakN17},
\textit{balanced incomplete block designs} (BIBDs)~\cite{Zauner99,ZhangG18},
\textit{$2$-transitive groups}~\cite{Creignou08},
\textit{semiregular divisible difference sets}~\cite{King16} and more generally \textit{difference families}~\cite{FickusMW21},
\textit{Latin squares}~\cite{ZhangG18},
and chains of alternating \textit{Naimark} and \textit{spatial complements}~\cite{CasazzaFMWZ11,FickusMW21}.
Other examples have been found numerically, and some of these have been perfected~\cite{ConwayHS96,DhillonHST08,CohnKM16,FuchsHS17}.
See~\cite{BachocBC04} for connections between ECTFFs and $t$-designs for Grassmannians, and~\cite{BachocE13} for various generalizations of ECTFFs.

Our work here is inspired by some ideas from the recent literature.
It turns out that some ETFs contain others:
if $\set{\bfphi_n}_{n=1}^N$ is any ETF for $\bbH$ then any subsequence of it is equiangular and might, on rare occasion, be a tight frame for the subspace of $\bbH$ that it spans.
See~\cite{FickusMJ16,ApplebyBDF17,FickusJKM18} for instances of this phenomenon.
Moreover, such \textit{sub-ETFs} can yield ECTFFs.
For example,
when an ETF partitions into regular simplices their respective spans form an ECTFF~\cite{FickusJKM18}.
This applies to \textit{Steiner} ETFs~\cite{GoethalsS70,FickusMT12},
certain \textit{polyphase} ETFs~\cite{FickusJMPW19} as well as to several infinite families of \textit{harmonic ETFs}~\cite{FickusJKM18,FickusS20},
namely ETFs that arise by restricting the characters of a finite abelian group to a \textit{difference set}~\cite{Konig99,StrohmerH03,XiaZG05,DingF07}.

In this paper we construct ECTFFs by exploiting new relationships between known ETFs.
In the next section we review some known concepts and results that we will need later on.
In Section~3, we define when a difference set for a finite abelian group is \textit{paired} with a difference set for its Pontryagin dual (Definition~\ref{def.paired difference sets}).
We show that a harmonic ETF that arises from such a pair contains many overlapping, unitarily equivalent copies of a smaller ETF, and moreover that the spans of these copies form an ECTFF (Theorem~\ref{thm.ECTFF from PDS}).
In Section~4 we exploit quadratic forms over $\bbF_2$ to construct an infinite family of paired difference sets (Theorem~\ref{thm.infinite family}).
For every integer $M\geq 2$ this yields an $\ETF(2^{M-1}(2^M\pm 1),2^{2M})$ that contains many copies of an $\ETF(\frac13(2^{2M}-1),2^{M-1}(2^M\mp 1))$.
These ETFs are not new: they equate to known families of \textit{strongly regular graphs}~\cite{Brouwer07,Brouwer17} via the correspondences of~\cite{HolmesP04,Waldron09,BargGOY15,FickusJMPW18}.
Moreover, their parameters match those of certain known real Steiner and Tremain~\cite{FickusJMP18} ETFs (or their Naimark complements).
That said, the resulting real $\ECTFF(2^{M-1}(2^M\pm 1),2^{2M},\frac13(2^{2M}-1))$ seem to be new except in the $(D,N,R)=(6,16,5)$ case.
They are not equi-isoclinic.
We conclude in Section~5 with some open problems concerning the existence of paired difference sets.

\section{Preliminaries}

Let $\calN$ be a finite set of cardinality $N>1$,
and let $\bbH$ be a Hilbert space over $\bbF$ (either $\bbR$ or $\bbC$) of dimension $D\geq 1$.
A sequence $\set{\calU_n}_{n\in\calN}$ of $R$-dimensional subspaces of $\bbH$ is a \textit{tight fusion frame} (TFF) for $\bbH$ if their projections $\set{\bfP_n}_{n\in\calN}$ satisfy $\sum_{n\in\calN}\bfP_n=A\bfI$ for some $A>0$.
This requires $A=\frac{NR}{D}\geq 1$ since $NR=\sum_{n\in\calN}\Tr(\bfP_n)=\Tr(\sum_{n\in\calN}\bfP_n)=\Tr(A\bfI)=AD$ and $NR=\sum_{n\in\calN}\rank(\bfP_n)\geq\rank(A\bfI)=D$.
Thus, for any $R$-dimensional subspaces $\set{\calU_n}_{n\in\calN}$ of $\bbH$,
\begin{equation}
\label{eq.ECTFF derivation 1}
0\leq\tfrac1{N(N-1)}\Tr\Bigbracket{\Bigparen{\sum_{n\in\calN}\bfP_n-\tfrac{NR}{D}\bfI}^2}
=\tfrac1{N(N-1)}\sum_{n_1\in\calN}\sum_{n_2\neq n_1}
\Tr(\bfP_{n_1}\bfP_{n_2})-\tfrac{R(NR-D)}{D(N-1)},
\end{equation}
where equality holds if and only if $\set{\calU_n}_{n\in\calN}$ is a TFF for $\bbH$.
Each term $\Tr(\bfP_{n_1}\bfP_{n_2})$ is real since
\begin{equation*}
[\dist(\calU_{n_1},\calU_{n_2})]^2
=\tfrac12\norm{\bfP_{n_1}-\bfP_{n_2}}_\Fro^2
=\tfrac12\Tr[(\bfP_{n_1}-\bfP_{n_2})^2]
=R-\Tr(\bfP_{n_1}\bfP_{n_2}).
\end{equation*}
As such, we can rearrange and continue~\eqref{eq.ECTFF derivation 1} as
\begin{equation}
\label{eq.generalized Welch}
\tfrac{R(NR-D)}{D(N-1)}
\leq\tfrac1{N(N-1)}\sum_{n_1\in\calN}\sum_{n_2\neq n_1}
\Tr(\bfP_{n_1}\bfP_{n_2})
\leq \max_{n_1\neq n_2}\Tr(\bfP_{n_1}\bfP_{n_2})
=R-\min_{n_1\neq n_2}[\dist(\calU_{n_1},\calU_{n_2})]^2.
\end{equation}
Equality holds throughout~\eqref{eq.generalized Welch} if and only if $\set{\calU_n}_{n\in\calN}$ is an ECTFF for $\bbH$,
namely a TFF that is also \textit{equichordal} in the sense that
$\dist(\calU_{n_1},\calU_{n_2})$ is constant over all $n_1\neq n_2$.
Rearranging~\eqref{eq.generalized Welch} gives the simplex bound~\eqref{eq.simplex bound}, which is called this since $\set{\calU_n}_{n\in\calN}$ is an ECTFF for $\bbH$ if and only if $\set{\bfP_n-\frac{R}{D}\bfI}_{n\in\calN}$ is a regular simplex for its span in the real Hilbert space of traceless self-adjoint operators on $\bbH$, equipped with the Frobenius inner product~\cite{ConwayHS96}.
In particular,
an ECTFF can only exist if $N\leq\rmd_\bbF(D)+1$ where $\rmd_\bbF(D)$ is the dimension of this space, namely
\begin{equation*}
\rmd_\bbF(D)=\left\{\begin{array}{cl}
\frac12(D-1)(D+2),&\ \bbF=\bbR,\\
D^2-1,&\ \bbF=\bbC.
\end{array}\right.
\end{equation*}
This necessary condition on the existence of an ECTFF is often called \textit{Gerzon's bound}; see~\cite{CalderbankHRSS99,KocakN17,ZhangG18} for some examples of ECTFFs that achieve equality in it.
When it is violated, $\set{\bfP_n-\frac{R}{D}\bfI}_{n\in\calN}$ cannot be mutually obtuse,
meaning there exists $n_1\neq n_2$ such that
\begin{equation*}
0\leq\ip{\bfP_{n_1}-\tfrac{R}{D}\bfI}{\bfP_{n_2}-\tfrac{R}{D}\bfI}_\Fro
=\Tr(\bfP_{n_1}\bfP_{n_2})-\tfrac{R^2}{D}
=\tfrac{R(D-R)}{D}-[\dist(\calU_{n_1},\calU_{n_2})]^2,
\end{equation*}
implying the \textit{orthoplex bound} of~\cite{ConwayHS96}, namely that
$\min_{n_1\neq n_2}\dist(\calU_{n_1},\calU_{n_2})
\leq\bigbracket{\tfrac{R(D-R)}{D}}^{\frac12}$;
see~\cite{KocakN17} for some recent constructions of sequences of subspaces that achieve equality in it.

Since both equichordality and tightness are preserved by both unitary transformations on $\bbH$ and bijections on $\calN$,
the existence of an ECTFF depends only on the parameters $(D,N,R)$ and $\bbF$.
We refer to any ECTFF for a possibly-complex $D$-dimensional Hilbert space $\bbH$ that consists of $N$ subspaces of it, each of dimension $R$, as an ``$\ECTFF(D,N,R)$,"
and say it is \textit{real} when $\bbH$ can be chosen to be $\bbR^D$.
The \textit{spatial complement}~\cite{CasazzaFMWZ11} of an $\ECTFF(D,N,R)$ $\set{\calU_n}_{n\in\calN}$ for $\bbH$ with $R<D$ is the sequence $\set{\calU_n^\perp}_{n\in\calN}$ of its members' orthogonal complements.
It is an $\ECTFF(D,N,D-R)$ for $\bbH$ since $\sum_{n\in\calN}(\bfI-\bfP_n)=(N-\tfrac{NR}D)\bfI$ and
\begin{equation*}
\dist(\calU_{n_1}^\perp,\calU_{n_2}^\perp)
=\tfrac1{\sqrt{2}}\norm{(\bfI-\bfP_{n_1})-(\bfI-\bfP_{n_2})}_\Fro
=\tfrac1{\sqrt{2}}\norm{\bfP_{n_1}-\bfP_{n_2}}_\Fro
=\dist(\calU_{n_1},\calU_{n_2}),
\ \forall\,n_1\neq n_2.
\end{equation*}
It can also be shown that if $\set{\calU_n}_{n\in\calN}$ achieves equality in the orthoplex bound then $\set{\calU_n^\perp}_{n\in\calN}$ does as well~\cite{King21}.

\subsection{Grassmannian codes and finite frame theory}
Equip $\bbF^\calN:=\set{\bfx:\calN\rightarrow\bbF}$ with the inner product $\ip{\bfx_1}{\bfx_2}:=\sum_{n\in\calN}\overline{\bfx_1(n)}\bfx_2(n)$.
(Under this notation, for any positive integer $N$, ``$\bbF^N$" is shorthand for $\bbF^{[N]}$ where $[N]:=\set{n\in\bbZ: 1\leq n\leq N}$.
Throughout, our complex inner products are conjugate-linear in their first arguments.)
The \textit{synthesis operator} of a sequence $\set{\bfphi_n}_{n\in\calN}$ of vectors in $\bbH$ is $\bfPhi:\bbF^\calN\rightarrow\bbH$, $\bfPhi\bfx:=\sum_{n\in\calN}\bfx(n)\bfphi_n$.
Its adjoint is the corresponding \textit{analysis operator} $\bfPhi^*:\bbH\rightarrow\bbF^\calN$, $\bfPhi^*\bfy=\sum_{n\in\calN}\ip{\bfphi_n}{\bfy}\bfdelta_n$,
where $\set{\bfdelta_n}_{n\in\calN}$ is the standard basis for $\bbF^\calN$.
In particular, we can regard a single vector $\bfphi\in\bbH$
as the synthesis operator $\bfphi:\bbF\rightarrow\bbH$, $\bfphi(x):=x\bfphi$ whose adjoint $\bfphi^*:\bbH\rightarrow\bbF$, $\bfphi^*\bfy=\ip{\bfphi_n}{\bfy}$ is a linear functional.
Composing $\bfPhi$ and $\bfPhi^*$ gives the \textit{frame operator}  $\bfPhi\bfPhi^*:\bbH\rightarrow\bbH$, $\bfPhi\bfPhi^*=\sum_{n\in\calN}\bfphi_n^{}\bfphi_n^*$
and the $\calN\times\calN$ \textit{Gram matrix} $\bfPhi^*\bfPhi:\bbF^\calN\rightarrow\bbF^\calN$ whose $(n,n')$th entry is $(\bfPhi^*\bfPhi)(n,n')=\ip{\bfphi_n}{\bfphi_{n'}}$.
In the special case where $\bbH=\bbF^\calD=\set{\bfy:\calD\rightarrow\bbF}$ for some finite set $\calD$ of cardinality $D>0$,
$\bfPhi$ is the $\calD\times\calN$ matrix whose $n$th column is $\bfphi_n$,
$\bfPhi^*$ is its $\calN\times\calD$ conjugate transpose,
and $\bfPhi\bfPhi^*$ and $\bfPhi^*\bfPhi$ are their $\calD\times\calD$ and $\calN\times\calN$ products, respectively.
In general, any $\bbF$-valued positive semidefinite $\calN\times\calN$ matrix $\bfG$ factors as $\bfG=\bfPhi^*\bfPhi$ for some sequence $\set{\bfphi_n}_{n\in\calN}$ of vectors in a Hilbert space $\bbH$ over $\bbF$ of dimension $D=\rank(\bfG)$.
This space is only unique up to a unitary transformation.

A sequence $\set{\bfphi_n}_{n\in\calN}$ of vectors in $\bbH$ is an ($A$-)\textit{tight frame} for $\bbH$ if $\bfPhi\bfPhi^*=A\bfI$ for some $A>0$.
In this case, any $\bfy\in\bbH$ can be written as
$\bfy=\frac1A\bfPhi\bfPhi^*\bfy=\frac1A\sum_{n\in\calN}\ip{\bfphi_n}{\bfy}\bfphi_n$ and so $\bbH$ is necessarily $\Span\set{\bfphi_n}_{n\in\calN}=\bfPhi(\bbF^\calD)$.
More generally, $\set{\bfphi_n}_{n\in\calN}$ is an $A$-tight frame for its span when $\bfPhi\bfPhi^*\bfy=A\bfy$ for all $\bfy\in\bfPhi(\bbF^\calD)$, namely when $\bfPhi\bfPhi^*\bfPhi=A\bfPhi$.
This occurs if and only if $(\bfPhi^*\bfPhi)^2=A\bfPhi^*\bfPhi$
(since having the latter implies that the image of $\bfPhi(\bfPhi^*\bfPhi-A\bfI)$ is contained in both $\bfPhi(\bbF^\calD)$ and $\ker(\bfPhi^*)=[\bfPhi(\bbF^\calD)]^\perp$).
As such,
a nonzero self-adjoint $\calN\times\calN$ matrix $\bfG$ is the Gram matrix $\bfPhi^*\bfPhi$ of an $A$-tight frame $\set{\bfphi_n}_{n\in\calN}$ for its span if and only if $\frac1A\bfG$ is a projection.
In this case, letting $D$ be $\dim(\Span\set{\bfphi_n}_{n\in\calN})=\rank(\bfPhi)=\frac1A\Tr(\bfG)$ we have that $\bfI-\frac1A\bfPhi^*\bfPhi$ is a projection of rank $N-D$.
If $D<N$, there thus exists an $A$-tight frame $\set{\bfpsi_n}_{n\in\calN}$ for a space of dimension $N-D$ that is uniquely defined (up to unitary transformations) by having
\begin{equation}
\label{eq.Naimark}
\bfPsi^*\bfPsi=A\bfI-\bfPhi^*\bfPhi,
\quad\text{i.e.,}\quad
\ip{\bfpsi_{n_1}}{\bfpsi_{n_2}}
=\left\{\begin{array}{rl}
A-\norm{\bfphi_n}^2,&\ n_1=n_2,\\
-\ip{\bfphi_{n_1}}{\bfphi_{n_2}},&\ n_1\neq n_2.
\end{array}\right.
\end{equation}
Such tight frames $\set{\bfphi_n}_{n\in\calN}$ and $\set{\bfpsi_n}_{n\in\calN}$ are called \textit{Naimark complements} of each other.

Now again let $\set{\calU_n}_{n\in\calN}$ be any sequence of $R$-dimensional subspaces of a $D$-dimensional Hilbert space $\bbH$.
For each $n\in\calN$ let $\bfPhi_n:\bbF^\calR\rightarrow\bbH$ be the synthesis operator of an ONB $\set{\bfphi_{n,r}}_{r\in\calR}$ for $\calU_n$,
and so $\bfP_n=\bfPhi_n^{}\bfPhi_n^*$ where $\bfPhi_n^*\bfPhi_n^{}=\bfI$.
Here, $\set{\bfphi_{n,r}}_{r\in\calR}$ is only unique up to $\calR\times\calR$ unitaries.
That is, it can be any member of the fiber of the \textit{Stiefel manifold} that projects onto the point $\calU_n$ in the Grassmannian.
The frame operator of the concatenation $\set{\bfphi_{n,r}}_{(n,r)\in\calN\times\calR}$ of these bases is
$\sum_{n\in\calN}\sum_{r\in\calR}\bfphi_{n,r}^{}\bfphi_{n,r}^*
=\sum_{n\in\calN}\bfPhi_n^{}\bfPhi_n^*
=\sum_{n\in\calN}\bfP_n$.
In particular,
$\set{\bfphi_{n,r}}_{(n,r)\in\calN\times\calR}$ is a tight frame for $\bbH$ if and only if $\set{\calU_n}_{n\in\calN}$ is a TFF for $\bbH$.
Meanwhile, the Gram matrix of $\set{\bfphi_{n,r}}_{(n,r)\in\calN\times\calR}$ has
$\ip{\bfphi_{n_1,r_1}}{\bfphi_{n_2,r_2}}
=\ip{\bfPhi_{n_1}\bfdelta_{r_1}}{\bfPhi_{n_2}\bfdelta_{r_2}}
=(\bfPhi_{n_1}^*\bfPhi_{n_2}^{})(r_1,r_2)$ as its $((n_1,r_1),(n_2,r_2))$th entry,
and so is naturally regarded as an $\calN\times\calN$ array whose $(n_1,n_2)$th block is the $\calR\times\calR$ \textit{cross-Gram} matrix $\bfPhi_{n_1}^*\bfPhi_{n_2}^{}$.
Since $\norm{\bfPhi_{n_1}^*\bfPhi_{n_2}}_\op\leq\norm{\bfPhi_{n_1}}_\op\norm{\bfPhi_{n_2}}_\op=1$,
the singular values of this matrix can be written as $\set{\cos(\theta_{n_1,n_2,r})}_{r=1}^R$ for some nondecreasing sequence $\set{\theta_{n_1,n_2,r}}_{r=1}^R$ of \textit{principal angles} in $[0,\frac\pi 2]$.
From this perspective,
$\set{\calU_n}_{n\in\calN}$ is equichordal if and only if
$\Tr(\bfP_{n_1}\bfP_{n_2})
=\Tr(\bfPhi_{n_1}^{}\bfPhi_{n_1}^*\bfPhi_{n_2}^{}\bfPhi_{n_2}^*)
=\Tr(\bfPhi_{n_2}^*\bfPhi_{n_1}^{}\bfPhi_{n_1}^*\bfPhi_{n_2}^{})
=\norm{\bfPhi_{n_1}^*\bfPhi_{n_2}^{}}_\Fro^2
=\sum_{r=1}^R\cos^2(\theta_{n_1,n_2,r})$
is constant over all $n_1\neq n_2$.
This perspective also gives a way to continue~\eqref{eq.ECTFF derivation 1} in a way that differs from~\eqref{eq.generalized Welch}:
\begin{equation*}
\tfrac{NR-D}{D(N-1)}
\leq\tfrac1{NR(N-1)}\sum_{n_1\in\calN}\sum_{n_2\neq n_1}
\norm{\bfPhi_{n_1}^*\bfPhi_{n_2}^{}}_\Fro^2\\
\leq\tfrac1{N(N-1)}\sum_{n_1\in\calN}\sum_{n_2\neq n_1}
\norm{\bfPhi_{n_1}^*\bfPhi_{n_2}^{}}_2^2
\leq \max_{n_1\neq n_2}\norm{\bfPhi_{n_1}^*\bfPhi_{n_2}^{}}_2^2.
\end{equation*}
Here, equality holds throughout if and only if $\set{\bfphi_{n,r}}_{(n,r)\in\calN\times\calR}$ is a tight frame for $\bbH$ and $\theta_{n_1,n_2,r}$ is constant over all $n_1\neq n_2$ and $r$.
This occurs if and only if $\set{\calU_n}_{n\in\calN}$ is an EITFF for $\bbH$:
since $\bfPhi_{n_1}^*\bfPhi_{n_1}^{}=\bfI$,
$\bfPhi_{n_1}^*\bfPhi_{n_2}^{}$ is a unitary scaled by some $\sigma\geq0$ if and only if
$\bfP_{n_1}\bfP_{n_2}\bfP_{n_1}
=\bfPhi_{n_1}^{}\bfPhi_{n_1}^*\bfPhi_{n_2}^{}\bfPhi_{n_2}^*\bfPhi_{n_1}^{}\bfPhi_{n_1}^*$
equals
$\bfPhi_{n_1}^{}(\sigma^2\bfI)\bfPhi_{n_1}^*
=\sigma^2\bfP_{n_1}$.
In particular, every EITFF is an optimal packing of members of the Grassmannian with respect to the \textit{spectral distance}, defined as
\smash{$\dist_{\mathrm{s}}(\calU_{n_1},\calU_{n_2})
:=(1-\norm{\bfPhi_{n_1}^*\bfPhi_{n_2}^{}}_2^2)^{\frac12}$}~\cite{DhillonHST08}.

In the special case where $\set{\calU_n}_{n\in\calN}$ is a sequence of subspaces of $\bbH$ of dimension $R=1$ we have $\bfPhi_n=\norm{\bfphi_n}^{-1}\bfphi_n$ where $\bfphi_n$ is an arbitrary nonzero vector in $\calU_n$.
Here, each cross-Gram matrix $\bfPhi_{n_1}^*\bfPhi_{n_2}^{}$ is a $1\times 1$ matrix with entry $(\norm{\bfphi_{n_1}}\norm{\bfphi_{n_2}})^{-1}\ip{\bfphi_{n_1}}{\bfphi_{n_2}}$.
In this case,
both the above inequality and \eqref{eq.generalized Welch} reduce to the Welch bound~\eqref{eq.Welch bound}.
Assuming without loss of generality that $\set{\bfphi_n}_{n\in\calN}$ is equal-norm,
it achieves equality in this bound if and only if it is a tight frame for $\bbH$ that is also \textit{equiangular} in the sense that $\abs{\ip{\bfphi_{n_1}}{\bfphi_{n_2}}}$ is constant over all $n_1\neq n_2$.

If $\set{\bfpsi_{n,r}}_{(n,r)\in\calN\times\calR}$ is the Naimark complement~\eqref{eq.Naimark} of a concatenation $\set{\bfphi_{n,r}}_{(n,r)\in\calN\times\calR}$ of ONBs of the subspaces $\set{\calU_n}_{n\in\calN}$ of an $\ECTFF(D,N,R)$ with $D<NR$ then
$\set{\calV_n}_{n\in\calN}$, $\calV_n:=\Span\set{\bfpsi_{n,r}}_{r\in\calR}$
is an $\ECTFF(NR-D,N,R)$ (and is an EITFF if and only if $\set{\calU_n}_{n\in\calN}$ is as well).
Taking alternating Naimark and spatial complements~\cite{CasazzaFMWZ11} of an $\ECTFF(D,N,R)$ often leads to an infinite chain of mutually distinct ECTFFs.
We caution that the spatial complement of an $\EITFF(D,N,R)$ is itself an EITFF if and only if $D=2R$.
In general, letting $\set{\bfPhi_n}_{n\in\calN}$ and $\set{\bfTheta_n}_{n\in\calN}$ be synthesis operators for ONBs for an ECTFF $\set{\calU_n}_{n\in\calN}$ and its spatial complement $\set{\calU_n^\perp}_{n\in\calN}$, respectively, we have $\bfPhi_n^*\bfPhi_n^{}=\bfI$, $\bfTheta_n^*\bfTheta_n^{}=\bfI$ and
$\bfPhi_n^{}\bfPhi_n^*+\bfTheta_n^{}\bfTheta_n^*=\bfI$ for all $n$, and so
\begin{align*}
\bfI-(\bfPhi_{n_1}^*\bfPhi_{n_2}^{})(\bfPhi_{n_1}^*\bfPhi_{n_2}^{})^*
&=\bfI-\bfPhi_{n_1}^*(\bfI-\bfTheta_{n_2}^{}\bfTheta_{n_2}^*)\bfPhi_{n_1}^{}
=(\bfPhi_{n_1}^*\bfTheta_{n_2}^{})(\bfPhi_{n_1}^*\bfTheta_{n_2}^{})^*,\\
\bfI-(\bfTheta_{n_1}^*\bfTheta_{n_2}^{})^*(\bfTheta_{n_1}^*\bfTheta_{n_2}^{})
&=\bfI-\bfTheta_{n_2}^*(\bfI-\bfPhi_{n_1}^{}\bfPhi_{n_1}^*)\bfTheta_{n_2}^{}
=(\bfPhi_{n_1}^*\bfTheta_{n_2}^{})^*(\bfPhi_{n_1}^*\bfTheta_{n_2}^{}).
\end{align*}
This implies that the sequences of singular values of $\bfPhi_{n_1}^*\bfPhi_{n_2}^{}$ and $\bfTheta_{n_1}^*\bfTheta_{n_2}^{}$ are $1$-padded versions of each other, in general.
In particular, any $\ECTFF(D,N,R)$ with $\frac D2<R<D$ is not an EITFF since some but not all of the principal angles between any two of its subspaces are $0$.

\subsection{Harmonic equiangular tight frames}

A \textit{character} of a finite abelian group $\calG$ is a homomorphism $\gamma:\calG\rightarrow\bbT:=\set{z\in\bbC: \abs{z}=1}$.
The set $\hat{\calG}$ of all characters of $\calG$ is called the \textit{Pontryagin dual} of $\calG$, and is itself a group under pointwise multiplication.
In this finite setting, it is well known that $\hat{\calG}$ is isomorphic to $\calG$ and that its members form an equal-norm orthogonal basis for $\bbC^\calG$.
The synthesis operator \smash{$\bfGamma:\bbC^{\hat{\calG}}\rightarrow\bbC^\calG$} of the sequence \smash{$\set{\gamma}_{\gamma\in\hat{\calG}}$} of all characters of $\calG$ (each serving as its own index) is thus a square \smash{$\calG\times\hat{\calG}$} matrix that satisfies $\bfGamma^*=N\bfGamma^{-1}$ where $N:=\#(\calG)$.
It is the \textit{character table} of $\calG$,
having $(g,\gamma)$th entry $\bfGamma(g,\gamma)=\gamma(g)$.
Its adjoint (conjugate-transpose) $\bfGamma^*$ is the analysis operator of the characters, namely the \textit{discrete Fourier transform} (DFT) over $\calG$.
We identify $\calG$ with the Pontryagin dual of $\hat{\calG}$ via the isomorphism $g\mapsto(\gamma\mapsto\gamma(g))$.
That is, we define $g(\gamma):=\gamma(g)$,
meaning the $\hat{\calG}\times\calG$ character table of $\hat{\calG}$ is simply the (nonconjugate) transpose of $\bfGamma$.

A \textit{harmonic} frame over $\calG$ is one obtained by restricting the characters of $\calG$ to some nonempty subset $\calD$ of $\calG$,
namely \smash{$\set{\bfphi_\gamma}_{\gamma\in\hat{\calG}}\subseteq\bbC^\calD$}, $\bfphi_\gamma(d):=\gamma(d)$.
It is a tight frame for $\bbC^\calD$ since its synthesis operator $\bfPhi$ satisfies
$(\bfPhi\bfPhi^*)(d_1,d_2)=(\bfGamma\bfGamma^*)(d_1,d_2)=N\bfI(d_1,d_2)$ for all $d_1,d_2\in\calD$.
It is also equal norm since $\norm{\bfphi_\gamma}^2=\sum_{d\in\calD}\abs{\gamma(d)}^2=D:=\#(\calD)$ for all $\gamma\in\hat{\calG}$.
Its Gram matrix is $\hat{\calG}$-circulant,
having entries arising from the DFT of the characteristic function $\bfchi_\calD$ of $\calD$:
\begin{equation*}
(\bfPhi^*\bfPhi)(\gamma_1,\gamma_2)
=\ip{\bfphi_{\gamma_1}}{\bfphi_{\gamma_2}}
=\sum_{g\in\calD}\overline{\gamma_1(g)}\gamma_2(g)
=\sum_{g\in\calG}\overline{(\gamma_1^{}\gamma_2^{-1})(g)}\bfchi_\calD(g)
=(\bfGamma^*\bfchi_\calD)(\gamma_1^{}\gamma_2^{-1}),
\end{equation*}
for any $\gamma_1,\gamma_2\in\hat{\calG}$.
To compute just the magnitudes of these entries,
we exploit the way in which the DFT interacts with the \textit{convolution} $\bfx_1*\bfx_2\in\bbC^\calG$, \smash{$(\bfx_1*\bfx_2)(g):=\sum_{g'\in\calG}\bfx_1(g')\bfx_2(g-g')$} and \textit{involution} $\tilde{\bfx}_1\in\bbC^\calG$,
\smash{$\tilde{\bfx}_1(g):=\overline{\bfx_1(-g)}$} of any given $\bfx_1,\bfx_2\in\bbC^\calG$.
(In this general setting, we typically use additive notation on $\calG$ and multiplicative notation on $\hat{\calG}$.)
Specifically, for any $\gamma\in\hat{\calG}$ we have
$[\bfGamma^*(\bfx_1*\bfx_2)](\gamma)=\bfx_1(\gamma)\bfx_2(\gamma)$
and $(\bfGamma^*\tilde{\bfx}_1)(\gamma)=\overline{(\bfGamma^*\bfx_1)(\gamma)}$.
Thus, for any $\gamma_1,\gamma_2\in\hat{\calG}$,
\begin{equation}
\label{eq.difference set derivation 1}
\abs{(\bfPhi^*\bfPhi)(\gamma_1,\gamma_2)}^2
=\abs{\ip{\bfphi_{\gamma_1}}{\bfphi_{\gamma_2}}}^2
=\abs{(\bfGamma^*\bfchi_\calD)(\gamma_1^{}\gamma_2^{-1})}^2
=[\bfGamma^*(\bfchi_\calD*\tilde{\bfchi}_\calD)](\gamma_1^{}\gamma_2^{-1}),
\end{equation}
where $\bfchi_\calD*\tilde{\bfchi}_\calD$ is the \textit{autocorrelation} function of $\bfchi_\calD$.
For any $g\in\calG$,
the mapping $g_1\mapsto(g_1,g_1-g)$ is a bijection from $\calD\cap(g+\calD)$ onto $\set{(g_1,g_2)\in\calD\times\calD: g=g_1-g_2}$,
meaning
\begin{equation*}
(\bfchi_\calD*\tilde{\bfchi}_\calD)(g)
=\sum_{g'\in\calG}\bfchi_\calD(g')\bfchi_{g+\calD}(g')
=\#[\calD\cap(g+\calD)]
=\#\set{(g_1,g_2)\in\calD\times\calD: g=g_1-g_2}
\end{equation*}
is both the number of elements of $\calG$ that $\calD$ has in common with $g+\calD$ and the number of distinct ways that $g$ can be written as a difference of members of $\calD$.

Now consider the special case where $\calD$ is a \textit{difference set} for $\calG$,
namely when $\calG\neq\set{0}$ and there exists $\Lambda$ such that $(\bfchi_\calD*\tilde{\bfchi}_\calD)(g)=\Lambda$ for all $g\neq0$.
Since $(\bfchi_\calD*\tilde{\bfchi}_\calD)(0)=D$,
this occurs if and only if $\bfchi_\calD*\tilde{\bfchi}_\calD=(D-\Lambda)\bfdelta_0+\Lambda\bfchi_\calG$.
Taking DFTs equivalently gives $\abs{(\bfGamma^*\bfchi_\calD)(\gamma)}^2=(D-\Lambda)+\Lambda N\bfdelta_1(\gamma)$ for all $\gamma\in\hat{\calG}$.
Here,
evaluating at $\gamma=1$ gives $D^2=(D-\Lambda)+\Lambda N$,
and so $\Lambda$ is necessarily \smash{$\frac{D(D-1)}{N-1}$},
a fact that also follows from a simple counting argument.
That is, $\calD$ is a difference set of $\calG$ if and only if
$\abs{(\bfGamma^*\bfchi_\calD)(\gamma)}^2=D-\tfrac{D(D-1)}{N-1}=\tfrac{D(N-D)}{N-1}$ for all $\gamma\neq1$.
When combined with~\eqref{eq.difference set derivation 1},
this classical characterization~\cite{Turyn65} of difference sets yields the more recent observation~\cite{Konig99,StrohmerH03,XiaZG05,DingF07} that a nonempty subset $\calD$ of $\calG$ is a difference set for $\calG$ if and only if the corresponding harmonic frame \smash{$\set{\bfphi_\gamma}_{\gamma\in\hat{\calG}}$} is an ETF for $\bbC^\calD$,
since it equates to having
\begin{equation*}
\tfrac{\abs{\ip{\bfphi_{\gamma_1}}{\bfphi_{\gamma_2}}}}{\norm{\bfphi_{\gamma_1}}\norm{\bfphi_{\gamma_2}}}
=\tfrac1D\abs{(\bfGamma^*\bfchi_\calD)(\gamma_1^{}\gamma_2^{-1})}
=\bigbracket{\tfrac{N-D}{D(N-1)}}^{\frac12},
\quad\forall\ \gamma_1\neq\gamma_2,
\end{equation*}
namely to achieving equality in the Welch bound~\eqref{eq.Welch bound}.
If $\calD$ is any difference set for $\calG$ then $\calD^\rmc$ is as well;
provided $\calD$ is not $\emptyset$ or $\calG$, the fact that $\bfGamma$ has equal-norm orthogonal columns implies that the two resulting harmonic ETFs are Naimark complements of each other.

\section{Equichordal tight fusion frames from paired difference sets}

As discussed in the previous section,
an $\ECTFF(D,N,R)$ for a $D$-dimensional space $\bbH$ equates to an $NR$-vector tight frame for $\bbH$ that can be partitioned into $N$ orthonormal subsequences whose $R\times R$ cross-Gram matrices have a common Frobenius norm,
and moreover, is an EITFF for $\bbH$ precisely when these cross-Gram matrices are a common scalar multiple of some unitaries.
This Stiefel-based perspective of Grassmannian codes pervades the literature:
\cite{Zauner99,King16} construct ECTFFs from collections of orthonormal vectors whose cross-Gram matrices are either remarkably sparse or flat (and so have readily computed Frobenius norms),
whereas~\cite{LemmensS73b,CalderbankTX15,Hoggar77,EtTaoui20,Waldron20,EtTaoui18,BlokhuisBE18} construct EITFFs by converting each off-diagonal entry of a suitably nice $\calN\times\calN$ matrix into an $\calR\times\calR$ scaled unitary,
all while simultaneously ensuring tightness.

That said, echoing a common theme of frame theory,
it is sometimes easier to find nice tight frames for the subspaces of an ECTFF than it is to find nice ONBs for them.
For example, when an ETF partitions into regular simplices,
their spans naturally form an ECTFF~\cite{FickusJKM18}.
To be fair, some of these ECTFFs are more easily constructed from other methods:
those arising from Steiner ETFs (including McFarland-harmonic ETFs~\cite{JasperMF14}) also arise directly from their underlying BIBDs~\cite{Zauner99},
while those arising from Singer-complement-harmonic ETFs are actually ETF-tensor-ONB-type EITFFs~\cite{FickusS20}.
Nevertheless, some of these ECTFFs have not been explained by competing methods,
including some arising from polyphase $\ETF(q+1,q^3+1)$ and twin-prime-power-complement-harmonic ETFs~\cite{FickusJKM18}.
One downside to such an approach is that it can become more difficult to characterize when a resulting ECTFF is actually an EITFF~\cite{FickusS20}.

In this paper we carry this idea further,
constructing ECTFFs from many overlapping sub-ETFs of a single ETF.
More precisely, we use a harmonic ETF that contains a sub-ETF whose members are themselves indexed by the elements of a difference set:

\begin{definition}
\label{def.paired difference sets}
We say a difference set $\calD$ for a finite abelian group $\calG$ is \textit{paired} with a difference set $\calE$ for its Pontryagin dual $\hat{\calG}$ if
$\set{\bfphi_\varepsilon}_{\varepsilon\in\calE}\subseteq\bbC^\calD$,
$\bfphi_\varepsilon(d):=\varepsilon(d)$ is a tight frame for its span.
\end{definition}

This concept was briefly discussed in~\cite{FickusMJ16}.
That paper also mentions two numerically obtained examples of such pairs.
The first of these consisted of certain subsets $\calD$ and $\calE$ of $\bbZ_2^4$ and its dual of cardinality $6$ and $10$, respectively.
The other consisted of two subsets of $\bbZ_4^2$ and its dual of these same cardinalities.
While the latter remains a mystery,
we were able to find an explicit version of the former that, as explained in Section~4,
generalizes to an infinite family of paired difference sets:

\begin{example}
\label{ex.PDS}
Let $\calG=\bbZ_2^4$ be the elementary abelian group of order $16$.
As detailed and generalized later on,
the function $\rmQ:\bbZ_2^4\rightarrow\bbZ_2$, $\rmQ(\bfx)=\rmQ(x_1,x_2,x_3,x_4)=x_1x_2+x_3x_4+x_3^2+x_4^2$ is a \textit{quadratic} form that gives rise to the \textit{symplectic} (nondegenerate alternating bilinear) form $\rmB:\bbZ_2^4\times\bbZ_2^4\rightarrow\bbZ_2$,
$\rmB(\bfx,\bfy)=\rmQ(\bfx+\bfy)+\rmQ(\bfx)+\rmQ(\bfy)=x_1y_2+x_2y_1+x_3y_4+x_4y_3$.
A point $\bfx\in\bbZ_2^4$ is \textit{singular} if $\rmQ(\bfx)=0$, and is otherwise \textit{nonsingular}.
Let $\calD$ and $\calE=\calD^\rmc$ be the $6$- and $10$-element sets of all singular and nonsingular points of $\rmQ$, respectively:
\begin{align}
\begin{split}
\label{eq.PDS(16,6,10)}
\calD&=\set{0000,0100,1000,1101,1110,1111},\\
\calE&=\set{0001,0010,0011,0101,0110,0111,1001,1010,1011,1100}.
\end{split}
\end{align}
These are complementary difference sets for $\bbZ_2^4$.
This can be verified by noting, for example, that every nonzero element of $\bbZ_2^4$ appears in the difference table of $\calD$ exactly $\Lambda=\frac{6(6-1)}{16-1}=2$ times:
\begin{equation*}
\begin{array}{c|cccccc}
   -&0000&0100&1000&1101&1110&1111\\\hline
0000&0000&0100&1000&1101&1110&1111\\
0100&0100&0000&1100&1001&1010&1011\\
1000&1000&1100&0000&0101&0110&0111\\
1101&1101&1001&0101&0000&0011&0010\\
1110&1110&1010&0110&0011&0000&0001\\
1111&1111&1011&0111&0010&0001&0000
\end{array}.
\end{equation*}
To explicitly construct the corresponding harmonic ETFs we identify $\bbZ_2^4$ with the Pontryagin dual via the isomorphism that maps $\bfy\in\bbZ_2^4$ to the character $\bfx\mapsto(-1)^{\rmB(\bfx,\bfy)}$.
Under this identification, the character table $\bfGamma$ becomes the following $\bbZ_2^4\times\bbZ_2^4$ matrix with entries $\bfGamma(\bfx,\bfy)=(-1)^{\rmB(\bfx,\bfy)}$,
and its $(\calD\times\bbZ_2^4)$- and $(\calE\times\bbZ_2^4)$-indexed submatrices $\bfGamma_0$ and $\bfGamma_1$ are the synthesis operators of a harmonic $\ETF(6,16)$ and its Naimark-complementary harmonic $\ETF(10,16)$, respectively:
\begin{equation}
\label{eq.16 x 16 Gamma}
\bfGamma=\left[\begin{smallmatrix}
+&+&+&+&+&+&+&+&+&+&+&+&+&+&+&+\\
+&+&-&-&+&+&-&-&+&+&-&-&+&+&-&-\\
+&-&+&-&+&-&+&-&+&-&+&-&+&-&+&-\\
+&-&-&+&+&-&-&+&+&-&-&+&+&-&-&+\\
+&+&+&+&+&+&+&+&-&-&-&-&-&-&-&-\\
+&+&-&-&+&+&-&-&-&-&+&+&-&-&+&+\\
+&-&+&-&+&-&+&-&-&+&-&+&-&+&-&+\\
+&-&-&+&+&-&-&+&-&+&+&-&-&+&+&-\\
+&+&+&+&-&-&-&-&+&+&+&+&-&-&-&-\\
+&+&-&-&-&-&+&+&+&+&-&-&-&-&+&+\\
+&-&+&-&-&+&-&+&+&-&+&-&-&+&-&+\\
+&-&-&+&-&+&+&-&+&-&-&+&-&+&+&-\\
+&+&+&+&-&-&-&-&-&-&-&-&+&+&+&+\\
+&+&-&-&-&-&+&+&-&-&+&+&+&+&-&-\\
+&-&+&-&-&+&-&+&-&+&-&+&+&-&+&-\\
+&-&-&+&-&+&+&-&-&+&+&-&+&-&-&+
\end{smallmatrix}\right],
\quad
\begin{array}{c}
\bfGamma_0=\left[\begin{smallmatrix}
+&+&+&+&+&+&+&+&+&+&+&+&+&+&+&+\\
+&+&+&+&+&+&+&+&-&-&-&-&-&-&-&-\\
+&+&+&+&-&-&-&-&+&+&+&+&-&-&-&-\\
+&+&-&-&-&-&+&+&-&-&+&+&+&+&-&-\\
+&-&+&-&-&+&-&+&-&+&-&+&+&-&+&-\\
+&-&-&+&-&+&+&-&-&+&+&-&+&-&-&+
\end{smallmatrix}\right],\bigskip\\
\bfGamma_1=\left[\begin{smallmatrix}
+&+&-&-&+&+&-&-&+&+&-&-&+&+&-&-\\
+&-&+&-&+&-&+&-&+&-&+&-&+&-&+&-\\
+&-&-&+&+&-&-&+&+&-&-&+&+&-&-&+\\
+&+&-&-&+&+&-&-&-&-&+&+&-&-&+&+\\
+&-&+&-&+&-&+&-&-&+&-&+&-&+&-&+\\
+&-&-&+&+&-&-&+&-&+&+&-&-&+&+&-\\
+&+&-&-&-&-&+&+&+&+&-&-&-&-&+&+\\
+&-&+&-&-&+&-&+&+&-&+&-&-&+&-&+\\
+&-&-&+&-&+&+&-&+&-&-&+&-&+&+&-\\
+&+&+&+&-&-&-&-&-&-&-&-&+&+&+&+
\end{smallmatrix}\right].
\end{array}
\end{equation}
(Here the elements of subsets of $\bbZ_2^4$ are ordered lexicographically, and ``$+$" and ``$-$" are shorthand for $1$ and $-1$, respectively.)
In particular, $\bfGamma_0^{}\bfGamma_0^*=16\bfI$ (tightness),
the diagonal entries of $\bfGamma_0^*\bfGamma_0^{}$ are $6$ while its off-diagonal entries have modulus $2$ (equiangularity),
and $\bfGamma_0^*\bfGamma_0^{}+\bfGamma_1^*\bfGamma_1^{}=16\bfI$
(Naimark complementarity).
Such real harmonic ETFs are well known~\cite{DingF07,JasperMF14},
and yield optimal packings of $16$ lines (one-dimensional subspaces) of $\bbR^6$ and $\bbR^{10}$.
What is new here is that under the aforementioned identification of $\bbZ_2^4$ with its Pontryagin dual,
the two difference sets $\calD$ and $\calE$ are paired in the sense of Definition~\ref{def.paired difference sets}.
That is, the columns of the $(\calD\times\calE)$-indexed submatrix
\begin{equation}
\label{eq.ETF(5,10)}
\bfGamma_{01}=\left[\begin{smallmatrix}
+&+&+&+&+&+&+&+&+&+\\
+&+&+&+&+&+&-&-&-&-\\
+&+&+&-&-&-&+&+&+&-\\
+&-&-&-&+&+&-&+&+&+\\
-&+&-&+&-&+&+&-&+&+\\
-&-&+&+&+&-&+&+&-&+
\end{smallmatrix}\right]
\end{equation}
of $\bfGamma_0$ (and $\bfGamma$) form a tight frame for their span.
This is far from obvious, but can be explicitly verified by showing that
$\bfGamma_{01}^{}\bfGamma_{01}^{*}\bfGamma_{01}^{}
=12\bfGamma_{01}^{}$, or equivalently,
that $\frac1{12}\bfGamma_{01}^{*}\bfGamma_{01}^{}$ is a projection.
Here, the tight frame constant $A=12$ is significant:
since $\frac1{12}\bfGamma_{01}^{*}\bfGamma_{01}^{}$ is a $10\times 10$ projection matrix with diagonal entries $\frac 6{12}$, the columns of $\bfGamma_{01}$ form a tight frame for a subspace of $\bbR^\calD\cong\bbR^6$ of dimension $\rank(\bfGamma_{01})
=\rank(\bfGamma_{01}^*\bfGamma_{01}^{})
=\Tr(\frac1{12}\bfGamma_{01}^{*}\bfGamma_{01}^{})
=5$.
As the $10$ columns of $\bfGamma_{01}$ are moreover equiangular (being $10$ of the $16$ equiangular columns of $\bfGamma_0$)
they thus form an $\ETF(5,10)$ for their span.
This itself is remarkable: there is an optimal packing of $10$ lines in $\bbR^5$ that extends to an optimal packing of $16$ lines in $\bbR^6$.
As we now explain, it moreover implies the existence of a real $\ECTFF(6,16,5)$ and a real $\ECTFF(10,16,5)$.
\end{example}

\begin{theorem}
\label{thm.ECTFF from PDS}
Let $\calD$ and $\calE$ be paired difference sets (Definition~\ref{def.paired difference sets}) for a finite abelian group $\calG$ and its Pontryagin dual $\hat{\calG}$, respectively.
For any $\gamma\in\hat{\calG}$ let $\calU_\gamma:=\Span\set{\bfphi_{\gamma\varepsilon}}_{\varepsilon\in\calE}$ where,
for any $\varepsilon\in\calE$,
$\bfphi_{\gamma\varepsilon}\in\bbC^\calD$ is defined by $\bfphi_{\gamma\varepsilon}(d):=\gamma(d)\varepsilon(d)$ for all $d\in\calD$.
Also let
\begin{equation}
\label{eq.paired diff set rank}
R=\tfrac{DE(N-1)}{(D+E-1)N-DE}
\ \text{where}\ D:=\#(\calD),\ E:=\#(\calE),\ N:=\#(\calG)=\#(\hat{\calG}).
\end{equation}
Then $\set{\calU_\gamma}_{\gamma\in\hat{\calG}}$ is an $\ECTFF(D,N,R)$ for $\bbC^\calD$ where, for each $\gamma\in\hat{\calG}$,
$\set{\bfphi_{\gamma\varepsilon}}_{\varepsilon\in\calE}$ is an $\ETF(R,E)$ for $\calU_\gamma$ that is unitarily equivalent to $\set{\bfphi_{\varepsilon}}_{\varepsilon\in\calE}$.
Moreover, the relation of being paired is symmetric: $\calE$ and $\calD$ are also paired, yielding an analogous $\ECTFF(E,N,R)$ for $\bbC^\calE$.
\end{theorem}

\begin{proof}
Since $\calD$ is a difference set for $\calG$ its harmonic frame \smash{$\set{\bfphi_\gamma}_{\gamma\in\hat{\calG}}$}, $\bfphi_\gamma(d):=\gamma(d)$ is an $\ETF(D,N)$ for $\bbC^\calD$.
Since $\calD$ and $\calE$ are paired,
the corresponding subsequence $\set{\bfphi_\varepsilon}_{\varepsilon\in\calE}$ of this $\ETF(D,N)$ is, by definition, a tight frame for $\calU_1=\Span\set{\bfphi_\varepsilon}_{\varepsilon\in\calE}$.
Moreover, since $\set{\bfphi_\gamma}_{\gamma\in\hat{\calG}}$ is equiangular,
this subsequence $\set{\bfphi_\varepsilon}_{\varepsilon\in\calE}$ is also equiangular,
and the two sequences share the same coherence (Welch bound).
In particular, $\set{\bfphi_\varepsilon}_{\varepsilon\in\calE}$ is an $\ETF(R,E)$ for $\calU_1$ where $R=\dim(\calU_1)$ satisfies
\begin{equation*}
(\tfrac{E}{R}-1)\tfrac1{E-1}
=\tfrac{E-R}{R(E-1)}
=\tfrac{N-D}{D(N-1)}.
\end{equation*}
Solving for $R$ gives~\eqref{eq.paired diff set rank}.
(In the degenerate case where $E=1$,
instead note that the single vector $\set{\bfphi_\varepsilon}_{\varepsilon\in\calE}$ is an ETF for its span, which has dimension \smash{$\tfrac{DE(N-1)}{(D+E-1)N-DE}=\tfrac{D(N-1)}{DN-D}=1=R$}.)
Next, for any $\gamma\in\hat{\calG}$,
$\set{\bfphi_{\gamma\varepsilon}}_{\varepsilon\in\calE}$ and $\set{\bfphi_\varepsilon}_{\varepsilon\in\calE}$ have the same Gram matrix, implying they are unitarily equivalent:
for any $\varepsilon_1,\varepsilon_2\in\calE$,
\begin{equation*}
\ip{\bfphi_{\gamma\varepsilon_1}}{\bfphi_{\gamma\varepsilon_2}}
=\sum_{d\in\calD}\overline{\gamma(d)\varepsilon_1(d)}\gamma(d)\varepsilon_2(d)
=\sum_{d\in\calD}\overline{\varepsilon_1(d)}\varepsilon_2(d)
=\ip{\bfphi_{\varepsilon_1}}{\bfphi_{\varepsilon_2}}.
\end{equation*}
In particular,
for any $\gamma\in\hat{\calG}$,
$\set{\bfphi_{\gamma\varepsilon}}_{\varepsilon\in\calE}$ is an $\ETF(R,E)$ for its span $\calU_\gamma$.

Next, to show that \smash{$\set{\calU_\gamma}_{\gamma\in\hat{\calG}}$} is an $\ECTFF(D,N,R)$ for $\bbC^\calD$,
let $\bfPhi_\gamma$ be the synthesis operator of $\set{\bfphi_{\gamma\varepsilon}}_{\varepsilon\in\calE}$.
Since $\set{\bfphi_{\gamma\varepsilon}}_{\varepsilon\in\calE}$ is an $E$-vector tight frame for $\calU_\gamma$ and $\norm{\bfphi_{\gamma\varepsilon}}^2=D$ for all $\varepsilon$,
the projection $\bfP_\gamma$ onto $\calU_\gamma$ can be expressed as
\smash{$\bfP_\gamma
=\frac{R}{DE}\bfPhi_\gamma^{}\bfPhi_\gamma^*$}.
To see that \smash{$\set{\calU_\gamma}_{\gamma\in\hat{\calG}}$} is a TFF for $\bbC^\calD$ note that for any $\gamma'\in\hat{\calG}$, there are exactly $E$ choices of $\gamma\in\hat{\calG}$ such that $\gamma'\in\gamma\calE$.
This allows us to write the fusion frame operator of \smash{$\set{\calU_\gamma}_{\gamma\in\hat{\calG}}$} in terms of the synthesis operator $\bfPhi$ of \smash{$\set{\bfphi_\gamma}_{\gamma\in\hat{\calG}}$}:
\begin{equation}
\label{eq.proof of ECTFF from PDS 1}
\sum_{\gamma\in\hat{\calG}}\bfP_\gamma
=\sum_{\gamma\in\hat{\calG}}\tfrac{R}{DE}\bfPhi_\gamma^{}\bfPhi_\gamma^*
=\tfrac{R}{DE}\sum_{\gamma\in\hat{\calG}}\sum_{\varepsilon\in\calE}
\bfphi_{\gamma\varepsilon}^{}\bfphi_{\gamma\varepsilon}^*
=\tfrac{R}{D}\sum_{\gamma'\in\hat{\calG}}\bfphi_{\gamma'}^{}\bfphi_{\gamma'}^{*}
=\tfrac{R}{D}\bfPhi\bfPhi^*
=\tfrac{NR}{D}\bfI.
\end{equation}
Next, to show that \smash{$\set{\calU_\gamma}_{\gamma\in\hat{\calG}}$} is equichordal,
note that for any $\gamma_1,\gamma_2\in\hat{\calG}$,
$\gamma_1\neq\gamma_2$,
\begin{equation*}
\Tr(\bfP_{\gamma_1}\bfP_{\gamma_2})
=\tfrac{R^2}{D^2E^2}\Tr(\bfPhi_{\gamma_1}^{}\bfPhi_{\gamma_1}^*\bfPhi_{\gamma_2}^{}\bfPhi_{\gamma_2}^*)
=\tfrac{R^2}{D^2E^2}\norm{\bfPhi_{\gamma_1}^*\bfPhi_{\gamma_2}^{}}_{\Fro}^2
=\tfrac{R^2}{D^2E^2}\sum_{\varepsilon_1\in\calE}\sum_{\varepsilon_2\in\calE}
\abs{\ip{\bfphi_{\gamma_1\varepsilon_1}}{\bfphi_{\gamma_2\varepsilon_2}}}^2.
\end{equation*}
Here, since \smash{$\set{\bfphi_\gamma}_{\gamma\in\hat{\calG}}$} is an $\ETF(D,N)$ with $\norm{\bfphi_\gamma}^2=D$ for all $\gamma$,
we have \smash{$\abs{\ip{\bfphi_{\gamma}}{\bfphi_{\gamma'}}}^2=\frac{D(N-D)}{N-1}$} for all $\gamma,\gamma'\in\hat{\calG}$ with $\gamma\neq\gamma'$.
As such, the value of the above sum depends entirely on the number of pairs $(\varepsilon_1,\varepsilon_2)\in\calE\times\calE$ such that $\gamma_1\varepsilon_1=\gamma_2\varepsilon_2$,
that is, such that $\gamma_1^{}\gamma_2^{-1}=\varepsilon_1^{-1}\varepsilon_2^{}$.
Since $\gamma_1\neq\gamma_2$ and $\calE$ is a difference set for $\hat{\calG}$,
this number is exactly \smash{$\frac{E(E-1)}{N-1}$}.
That is, for any $\gamma_1\neq\gamma_2$,
\begin{equation*}
\Tr(\bfP_{\gamma_1}\bfP_{\gamma_2})
=\tfrac{R^2}{D^2E^2}\sum_{\varepsilon_1\in\calE}\sum_{\varepsilon_2\in\calE}
\abs{\ip{\bfphi_{\gamma_1\varepsilon_1}}{\bfphi_{\gamma_2\varepsilon_2}}}^2
=\tfrac{R^2}{D^2E^2}\bigset{\tfrac{E(E-1)}{N-1}D^2+[E^2-\tfrac{E(E-1)}{N-1}]\tfrac{D(N-D)}{N-1}}.
\end{equation*}
Since this value is constant over all $\gamma_1\neq\gamma_2$,
\smash{$\set{\calU_\gamma}_{\gamma\in\hat{\calG}}$} is an ECTFF.
(Alternatively, we may forgo~\eqref{eq.proof of ECTFF from PDS 1} provided we instead use~\eqref{eq.paired diff set rank} to show that the above value for $\Tr(\bfP_{\gamma_1}\bfP_{\gamma_2})$ simplifies to \smash{$\frac{R(NR-D)}{D(N-1)}$},
meaning $\set{\calU_\gamma}_{\gamma\in\hat{\calG}}$ achieves equality throughout~\eqref{eq.generalized Welch} and so is necessarily tight.)

For the final conclusions,
recall that in general,
we identify $\calG$ with the Pontryagin dual of $\hat{\calG}$ via the isomorphism $g\mapsto(\gamma\mapsto\gamma(g))$, that is, we define $g(\gamma):=\gamma(g)$.
Since $\calE$ is a difference set for $\hat{\calG}$,
its harmonic frame $\set{\bfpsi_g}_{g\in\calG}$, $\bfpsi_g(\varepsilon):=g(\varepsilon)=\varepsilon(g)$
is an $\ETF(E,N)$ for $\bbC^\calE$.
To show that $\calE$ and $\calD$ are paired,
we show the corresponding subsequence $\set{\bfpsi_d}_{d\in\calD}$ of this harmonic ETF is a tight frame for its span.
To do this, note the synthesis operators $\bfPhi_{\calE}$ and $\bfPsi_{\calD}$ of $\set{\bfphi_\varepsilon}_{\varepsilon\in\calE}$ and $\set{\bfpsi_d}_{d\in\calD}$, respectively,
are transposes of each other: for any $d\in\calD$, $\varepsilon\in\calE$,
\begin{equation*}
\bfPsi_{\calD}(\varepsilon,d)
=\bfpsi_d(\varepsilon)
=d(\varepsilon)=\varepsilon(d)
=\bfphi_\varepsilon(d)
=\bfPhi_{\calE}(d,\varepsilon).
\end{equation*}
At the same time, since $\calD$ and $\calE$ are paired, $\set{\bfphi_\varepsilon}_{\varepsilon\in\calE}$ is a tight frame for its span and so there exists $A>0$ such that
$\bfPhi_{\calE}^{}\bfPhi_{\calE}^{*}\bfPhi_{\calE}^{}=A\bfPhi_{\calE}^{}$.
Taking transposes of this equation thus gives
\begin{equation*}
A\bfPsi_{\calD}
=A\bfPhi_\calE^\rmT
=(\bfPhi_{\calE}^{}\bfPhi_{\calE}^{*}\bfPhi_{\calE}^{})^\rmT
=\bfPhi_{\calE}^{\rmT}(\bfPhi_{\calE}^*)^{\rmT}\bfPhi_{\calE}^{\rmT}
=\bfPhi_{\calE}^{\rmT}(\bfPhi_{\calE}^\rmT)^{*}\bfPhi_{\calE}^{\rmT}
=\bfPsi_{\calD}^{}\bfPsi_{\calD}^{*}\bfPsi_{\calD}^{},
\end{equation*}
and so $\set{\bfpsi_d}_{d\in\calD}$ is indeed a tight frame for its span.
Since the expression given for $R$ in~\eqref{eq.paired diff set rank} is symmetric with respect to $D$ and $E$, this span has dimension $R$.
As such, applying the first part of this theorem to it yields an $\ECTFF(E,N,R)$ for $\bbC^\calE$, as claimed.
\end{proof}

\begin{example}
\label{ex.ECTFF}
When applied to the paired difference sets $\calD$ and $\calE$ of \eqref{eq.PDS(16,6,10)} of Example~\ref{ex.PDS},
Theorem~\ref{thm.ECTFF from PDS} implies that the columns of $\bfGamma_{01}$ of~\eqref{eq.ETF(5,10)} form an $\ETF(5,10)$ for their span,
that the $\ETF(6,16)$ formed by the columns of $\bfGamma_0$ of~\eqref{eq.16 x 16 Gamma} contains sixteen unitarily equivalent copies of this $\ETF(5,10)$ (each indexed by a shift $\bfx+\calE$ of $\calE$),
and that the spans of these copies form an $\ECTFF(6,16,5)$ for $\bbC^\calD$.
In fact, since the character table $\bfGamma$ of~\eqref{eq.16 x 16 Gamma} of $\bbZ_2^4$ is real-valued,
this ECTFF is real.
The existence of a real $\ECTFF(6,16,5)$ is not new:
one arises, for example, by taking the spatial complement of a real $\ECTFF(6,16,1)$ that equates to a real $\ETF(6,16)$.

Theorem~\ref{thm.ECTFF from PDS} further gives that $\calE$ and $\calD$ are paired,
and this yields a seemingly new ECTFF.
As seen in its proof, this stems from the fact that the character tables of a finite abelian group and its Pontryagin dual are transposes of each other,
and the fact that the columns of a matrix form a tight frame for their span if and only if their rows do as well.
In general, this means that the $\ECTFF(E,N,R)$ produced by Theorem~\ref{thm.ECTFF from PDS} arises as the row spaces of the various $[(g+\calD)\times\calE)]$-indexed submatrices of $\bfGamma$.
For this particular example,
we further have that $\calE=\calD^\rmc$ where $\bbZ_2^4$ has been identified with its Pontryagin dual in a way that makes $\bfGamma$ of~\eqref{eq.16 x 16 Gamma} symmetric.
As such, Theorem~\ref{thm.ECTFF from PDS} moreover gives here that the columns of the $(\calE\times\calD)$-indexed submatrix of $\bfGamma$ (the transpose of~\eqref{eq.ETF(5,10)}) form an $\ETF(5,6)$ for their span,
that the $\ETF(10,16)$ formed by the columns of $\bfGamma_1$ of~\eqref{eq.16 x 16 Gamma} (the Naimark complement of that formed by the columns of $\bfGamma_0$) contains sixteen unitarily equivalent copies of this $\ETF(5,6)$ (each indexed by a shift $\calD$),
and that the spans of these copies form a real $\ECTFF(10,16,5)$ for $\bbR^\calE$.
We know of no other construction of an ECTFF with these parameters (real or complex).
\end{example}

In general,
when paired difference sets $\calD$ and $\calE$ for $\calG$ and $\hat{\calG}$ exist, they are not unique.
For example,
Theorem~\ref{thm.ECTFF from PDS} gives that for any $\gamma\in\hat{\calG}$, $\set{\bfphi_{\gamma\varepsilon}}_{\varepsilon\in\calE}$ is an ETF (and so a tight frame) for its span, implying $\calD$ and $\gamma\calE$ are also paired.
Since Theorem~\ref{thm.ECTFF from PDS} also gives that ``pairing" is a symmetric relation,
we further have that $g+\calD$ and $\gamma\calE$ are paired regardless of one's choice of $g\in\calG$, $\gamma\in\hat{\calG}$.
One can also apply an automorphism $\sigma$ of $\calG$ to $\calD$ provided one simultaneously applies the induced automorphism $\gamma\mapsto\gamma\circ\sigma^{-1}$ to $\hat{\calG}$.
This is because the $[\sigma(\calD)\times(\calE\circ\sigma^{-1})]$-indexed submatrix of $\bfGamma$ can be obtained by pre- and post-multiplying its $(\calD\times\calE)$-indexed submatrix by permutation matrices,
implying its columns form a tight frame for their span.

The next result gives the most fundamental characterization of paired difference sets that we have found so far.
It is not purely combinatorial,
but rather involves sums over certain submatrices of the character table.
One can use it, for example, to obtain alternate proofs of the above facts.

\begin{theorem}
Let $\calD$ and $\calE$ be nonempty difference sets for a finite abelian group $\calG$ and its Pontryagin dual $\hat{\calG}$, respectively.
Then $\calD$ and $\calE$ are paired (Definition~\ref{def.paired difference sets}) if and only if
\begin{equation}
\label{eq.partial character table sum}
\sum\nolimits_{(d',\varepsilon')\in(\calD-d)\times(\varepsilon^{-1}\calE)}\varepsilon'(d')
\end{equation}
has constant value over all $d\in\calD$, $\varepsilon\in\calE$.
\end{theorem}

\begin{proof}
Let $\bfPhi$ be the $(\calD\times\calE)$-indexed submatrix of the character table of $\calG$, which is defined by having $\bfPhi(d,\varepsilon)=\varepsilon(d)$ for all $d\in\calD$, $\varepsilon\in\calE$.
By definition, $\calD$ and $\calE$ are paired if and only if $\bfPhi\bfPhi^*\bfPhi=A\bfPhi$ for some $A>0$.
Since $\bfPhi\neq\bfzero$,
this occurs if and only if $\bfPhi\bfPhi^*\bfPhi=A\bfPhi$ for some $A\in\bbC$:
in the latter case, we have $(\bfPhi^*\bfPhi)^2=A(\bfPhi^*\bfPhi)$,
meaning $A$ is the nonzero eigenvalue of the nonzero positive semidefinite matrix $\bfPhi^*\bfPhi$.
This equates to having $A\in\bbC$ such that
\begin{equation*}
A\varepsilon(d)
=A\bfPhi(d,\varepsilon)
=\sum_{d'\in\calD}\sum_{\varepsilon'\in\calE}\bfPhi(d,\varepsilon')\bfPhi^*(\varepsilon',d')\bfPhi(d',\varepsilon)
=\sum_{d'\in\calD}\sum_{\varepsilon'\in\calE}\varepsilon'(d)\overline{\varepsilon'(d')}\varepsilon(d')
\end{equation*}
for all $d\in\calD$, $\varepsilon\in\calE$.
Simplifying,
this occurs if and only if for some $A\in\bbC$,
\begin{equation*}
\overline{A}
=\sum_{d'\in\calD}\sum_{\varepsilon'\in\calE}\overline{\varepsilon'(d)}\varepsilon'(d')\overline{\varepsilon(d')}\varepsilon(d)
=\sum_{d'\in\calD}\sum_{\varepsilon'\in\calE}(\varepsilon'\varepsilon^{-1})(d'-d)
\end{equation*}
for all $d\in\calD$, $\varepsilon\in\calE$.
Making a change of variables gives that this equates to the value of~\eqref{eq.partial character table sum} being constant over all $d\in\calD$, $\varepsilon\in\calE$.
\end{proof}

Some paired difference sets $\calD$ and $\calE$ are \textit{trivial} in the sense that either $\calD$ or $\calE$ is either a singleton set or is its entire group.
In such cases, Theorem~\ref{thm.ECTFF from PDS} still applies, but the resulting ETFs and ECTFFs are not new.
To explain,
any singleton set $\calD$ is a difference set for $\calG$ and it pairs with any nonempty difference set $\calE$ for $\hat{\calG}$ since $\set{\bfphi_\varepsilon}_{\varepsilon\in\calE}$ equates to a sequence of $E$ unimodular scalars in the one-dimensional space $\bbC^\calD\cong\bbC$.
Also, \smash{$\calE=\hat{\calG}$} is a difference set for $\hat{\calG}$ that pairs with any nonempty difference set $\calD$ for $\calG$ since \smash{$\set{\bfphi_\varepsilon}_{\varepsilon\in\calE}$} is the harmonic ETF arising from $\calD$.
In either of these two cases the resulting $\ECTFF(D,N,D)$ \smash{$\set{\calU_\gamma}_{\gamma\in\hat{\calG}}$} consists of $N$ copies of the entire space $\bbC^\calD$.
Meanwhile, any singleton set $\calE$ is a difference set for \smash{$\hat{\calG}$}
and it pairs with any nonempty difference set $\calD$ for $\calG$ since the single vector $\set{\bfphi_{\varepsilon}}_{\varepsilon\in\calE}$ is a tight frame for its one-dimensional span.
In this case, the $\ECTFF(D,N,1)$ \smash{$\set{\calU_\gamma}_{\gamma\in\hat{\calG}}$} consists of the $N$ one-dimensional subspaces of $\bbC^\calD$ that are individually spanned by the members of the underlying harmonic $\ETF(D,N)$ \smash{$\set{\bfphi_\gamma}_{\gamma\in\hat{\calG}}$} for $\bbC^\calD$.

The remaining ``trivial" case is the most interesting in that it yields a nontrivial result:
$\calD=\calG$ is a difference set for $\calG$ that pairs with any nonempty difference set $\calE$ for $\hat{\calG}$ since in this case $\set{\bfphi_\varepsilon}_{\varepsilon\in\calE}$ is a sequence of equal-norm orthogonal vectors (and so is a tight frame for its span).
In this case, for any $\gamma\in\hat{\calG}$,
the corresponding subspace $\calU_\gamma=\Span\set{\bfphi_{\gamma\varepsilon}}_{\varepsilon\in\calE}$ of $\bbC^\calD=\bbC^\calG$ is the span of the characters of $\calG$ that happen to lie in $\gamma\calE:=\set{\gamma\varepsilon: \varepsilon\in\calE}$.
Taking the DFT of $\calU_\gamma$ thus yields the subspace of \smash{$\bbC^{\hat{\calG}}$} that is spanned by the members of the standard basis that are indexed by $\gamma\calE$,
namely
$\bfGamma^*\calU_\gamma
=\Span\set{\bfGamma^*\bfphi_{\gamma\varepsilon}}_{\varepsilon\in\calE}
=\Span\set{\bfdelta_{\gamma\varepsilon}}_{\varepsilon\in\calE}
=\Span\set{\bfdelta_{\gamma'}}_{\gamma'\in\gamma\calE}$.
Since the DFT is a scalar multiple of a unitary operator,
$\set{\calU_\gamma}_{\gamma\in\hat{\calG}}$ is an ECTFF for $\bbC^{\calD}=\bbC^\calG$ if and only if $\set{\bfGamma^*\calU_\gamma}_{\gamma\in\hat{\calG}}$ is an ECTFF for \smash{$\bbC^{\hat{\calG}}$}.
While Theorem~\ref{thm.ECTFF from PDS} gives the former,
the latter is known:
Zauner~\cite{Zauner99} showed that if $(\calV,\calB)$ is any BIBD,
then the subspaces $\set{\calX_v}_{v\in\calV}$ of $\bbR^\calB$,
$\calX_v:=\Span\set{\bfdelta_b: v\in b\in\calB}$ form an ECTFF for $\bbR^\calB$.
Here, since $\calE$ is a difference set of $\hat{\calG}$,
$\calE^{-1}:=\set{\varepsilon^{-1}: \varepsilon\in\calE}$ is as well,
implying $\calB=\set{\gamma\calE^{-1}}_{\gamma\in\hat{\calG}}$ is the block set for a BIBD on \smash{$\calV=\hat{\calG}$}.
In this case, the $\gamma$th subspace of Zauner's construction is
\smash{$\calX_\gamma
=\Span\set{\bfdelta_{\gamma'}: \gamma\in\gamma'\calE^{-1}}
=\Span\set{\bfdelta_{\gamma'}}_{\gamma'\in\gamma\calE}
=\bfGamma^*\calU_\gamma$}.

When $\calD$ and $\calE$ are nontrivial paired difference sets the $\ECTFF(D,N,R)$ $\set{\calU_\gamma}_{\gamma\in\hat{\calG}}$ constructed in Theorem~\ref{thm.ECTFF from PDS} consists of $N$ proper and distinct subspaces of \smash{$\bbC^\calD$}.
Indeed,
having $1<D$ and $E<N$ implies
\smash{$R
=\tfrac{DE(N-1)}{(D+E-1)N-DE}
<D$} and so
\smash{$[\dist(\calU_{\gamma_1},\calU_{\gamma_2})]^2
=\tfrac{R(D-R)}{D}\tfrac{N}{N-1}>0$} for any $\gamma_1\neq\gamma_2$.
Moreover, in this nontrivial case, these subspaces are not equi-isoclinic since they have nontrivial intersection:
since $\calE$ is a difference set for $\hat{\calG}$,
\smash{$\#(\gamma_1\calE\cap\gamma_2\calE)
=\#[\calE\cap(\gamma_1^{-1}\gamma_2^{})\calE]
=\frac{E(E-1)}{N-1}>0$} for any $\gamma_1\neq\gamma_2$,
implying
$\calU_{\gamma_1}\cap\calU_{\gamma_2}
=\Span\set{\bfphi_\gamma}_{\gamma\in\gamma_1\calE}
\cap\Span\set{\bfphi_\gamma}_{\gamma\in\gamma_1\calE}
\supseteq\Span\set{\bfphi_\gamma}_{\gamma\in\gamma_1\calE\cap\gamma_2\calE}$ has positive dimension.
As such, some but not all of the principal angles between $\calU_{\gamma_1}$ and $\calU_{\gamma_2}$ are zero.
We further note that if $\calD$ and $\calE$ are nontrivial paired difference sets then so are $\calE$ and $\calD$, implying the $N$ subspaces of the resulting $\ECTFF(E,N,R)$ are also proper, distinct, and not equi-isoclinic.

We conclude this section with a necessary condition on the existence of paired difference sets, and a characterization of when the complements of paired difference sets are themselves paired:

\begin{theorem}
\label{thm.necessary}
Let $\calD$ and $\calE$ be paired difference sets (Definition~\ref{def.paired difference sets}) for a finite abelian group $\calG$ and its Pontryagin dual $\hat{\calG}$, respectively, where $\calD\neq\calG$ and $\calE\neq\hat{\calG}$.
Then
\begin{equation}
\label{eq.necessary}
D+E\leq N
\ \text{where}\ D:=\#(\calD),\ E:=\#(\calE),\ N:=\#(\calG)=\#(\hat{\calG}).
\end{equation}
Moreover, $\calD^\rmc$ and $\calE^\rmc$ are paired difference sets if and only if $D+E=N$.
\end{theorem}

\begin{proof}
Here,
for any nonempty subsets $\calX$ and $\calY$ of $\calG$ and $\hat{\calG}$, respectively,
we denote the corresponding submatrix of the $\calG\times\hat{\calG}$ character table $\bfGamma$ of $\calG$ as $\bfGamma_{\calX\times\calY}$.
Since $\calD$ and $\calE$ are paired difference sets,
$\bfGamma_{\calD\times\calE}$ is the synthesis operator of an $\ETF(R,E)$ for its span  where $R$ is given by~\eqref{eq.paired diff set rank}.
In particular, the spectrum of $\bfGamma_{\calD\times\calE}^{}\bfGamma_{\calD\times\calE}^{*}$ consists of a tight frame constant $A>0$ and $0$ with multiplicities $R$ and $D-R$, respectively,
while the spectrum of $\bfGamma_{\calD\times\calE}^{*}\bfGamma_{\calD\times\calE}^{}$ consists of $A$ and $0$ with multiplicities $R$ and $E-R$, respectively.
Here, since every entry of $\bfGamma_{\calD\times\calE}$ has unit modulus,
$\bfGamma_{\calD\times\calE}^{*}\bfGamma_{\calD\times\calE}$ is an $E\times E$ matrix whose every diagonal entry has value $D$,
implying
$DE=\Tr(\bfGamma_{\calD\times\calE}^{*}\bfGamma_{\calD\times\calE}^{})=RA$, and so $A=\frac{DE}{R}$.
We now claim that $A\neq N$.
Indeed, otherwise~\eqref{eq.paired diff set rank} gives
\smash{$N
=A
=\tfrac{DE}{R}
=\tfrac{(D+E-1)N-DE}{N-1}$},
which implies the following contradiction of the assumption that $\calD\neq\calG$ and \smash{$\calE\neq\hat{\calG}$}:
\begin{equation*}
0=N(N-1)-[(D+E-1)N-DE]=N^2-(D+E)N+DE=(N-D)(N-E)>0.
\end{equation*}
Next note that since $\bfGamma$ is a (possibly complex) Hadamard matrix,
\begin{equation*}
N\bfI_{\calD\times\calD}
=\bfGamma_{\calD\times\hat{\calG}}^{}\bfGamma_{\calD\times\hat{\calG}}^{*}
=\bfGamma_{\calD\times\calE}^{}\bfGamma_{\calD\times\calE}^{*}
+\bfGamma_{\calD\times\calE^\rmc}^{}\bfGamma_{\calD\times\calE^\rmc}^{*},
\end{equation*}
and so the spectrum of $\bfGamma_{\calD\times\calE^\rmc}^{}\bfGamma_{\calD\times\calE^\rmc}^{*}
=N\bfI_{\calD\times\calD}-\bfGamma_{\calD\times\calE}^{}\bfGamma_{\calD\times\calE}^{*}$ consists of $N-A\neq0$ and $N$ with multiplicities $R$ and $D-R$, respectively.
In particular, $\bfGamma_{\calD\times\calE^\rmc}^{}\bfGamma_{\calD\times\calE^\rmc}^{*}$ is invertible, implying
\begin{equation*}
D
=\rank(\bfGamma_{\calD\times\calE^\rmc}^{}\bfGamma_{\calD\times\calE^\rmc}^{*})
=\rank(\bfGamma_{\calD\times\calE^\rmc}^{})
\leq\#(\calE^\rmc)
=N-E,
\end{equation*}
namely~\eqref{eq.necessary}.
For the final conclusion, note that since $D\leq N-E$,
the spectrum of
$\bfGamma_{\calD\times\calE^\rmc}^{*}\bfGamma_{\calD\times\calE^\rmc}^{}$ is obtained by padding that of $\bfGamma_{\calD\times\calE^\rmc}^{}\bfGamma_{\calD\times\calE^\rmc}^{*}$ with $N-D-E$ zeros,
and so consists of $N-A\neq0$, $N$ and $0$ with multiplicities $R$, $D-R$ and $N-D-E$, respectively.
Since
\begin{equation*}
N\bfI_{\calE^\rmc\times\calE^\rmc}
=\bfGamma_{\calG\times\calE^\rmc}^{*}\bfGamma_{\calG\times\calE^\rmc}^{}
=\bfGamma_{\calD\times\calE^\rmc}^{*}\bfGamma_{\calD\times\calE^\rmc}^{}
+\bfGamma_{\calD^\rmc\times\calE^\rmc}^{*}\bfGamma_{\calD^\rmc\times\calE^\rmc}^{},
\end{equation*}
this in turn implies that the spectrum of
$\bfGamma_{\calD^\rmc\times\calE^\rmc}^{*}\bfGamma_{\calD^\rmc\times\calE^\rmc}^{}
=N\bfI_{\calE^\rmc\times\calE^\rmc}
-\bfGamma_{\calD\times\calE^\rmc}^{*}\bfGamma_{\calD\times\calE^\rmc}^{}$ consists of $A$, $0$ and $N$ with multiplicities $R$, $D-R$ and $N-D-E$, respectively.
In particular, since $A\neq N$,
we have that $D+E=N$ if and only if
$\bfGamma_{\calD^\rmc\times\calE^\rmc}^{*}\bfGamma_{\calD^\rmc\times\calE^\rmc}^{}$ is a scalar multiple of a projection,
namely if and only if $\calD^\rmc$ and $\calE^\rmc$ are paired difference sets.
\end{proof}

When $\calD$ and $\calE$ are paired difference sets that achieve equality in~\eqref{eq.necessary}, the cardinalities of $\calD^\rmc$ and $\calE^\rmc$ equal those of $\calE$ and $\calD$, respectively.
As such, even in this case, the parameters of the ECTFF that arises via Theorem~\ref{thm.ECTFF from PDS} from $\calD^\rmc$ and $\calE^\rmc$ equal those of the ECTFF that arises via Theorem~\ref{thm.ECTFF from PDS} from $\calE$ and $\calD$ directly.
We further note that the inequality in~\eqref{eq.necessary} is sometimes strict, including for example trivial paired difference sets with $D=1<N$ and $E<N-1$.

\section{Paired difference sets from quadratic forms}

In this section we construct an infinite family of nontrivial paired difference sets.
Applying Theorem~\ref{thm.ECTFF from PDS} to them produces two infinite families of ECTFFs.
As foreshadowed by Example~\ref{ex.PDS},
the construction involves quadratic forms on finite vector spaces over the binary field $\bbF_2$.
Quadratic forms are a classical subject,
with those over fields of characteristic two being notably different than their cousins over other fields~\cite{Grove02}.
From the perspective of finite frame theory,
quadratic forms over $\bbF_2$ are already notable,
having been used to construct the maximum number of \textit{mutually unbiased bases} in real spaces whose dimension is a power of $4$~\cite{CameronS73,CalderbankCKS97}.
We review only the small part of the classical literature~\cite{Taylor92,Grove02} that we use to prove our results.

Let $\calV$ be a finite-dimensional vector space over $\bbF_2$,
which equates to a finite elementary abelian $2$-group.
A \textit{bilinear form} on $\calV$ is any function $\rmB:\calV\times\calV\rightarrow\bbF_2$ that is linear in either argument while the other is held fixed.
Such a form is \textit{nondegenerate} if having $\rmB(v_1,v_2)=0$ for all $v_1\in\calV$ implies that $v_2=0$.
Any choice of nondegenerate bilinear form $\rmB$ on $\calV$ induces an isomorphism from (the additive group of) $\calV$ to its Pontryagin dual, namely the mapping $v_2\mapsto(v_1\mapsto(-1)^{\rmB(v_1,v_2)})$.
(This mapping is a well-defined homomorphism since $\rmB$ is bilinear,
and is injective since $\rmB$ is nondegenerate.)
Under this identification, the character table $\bfGamma$ of $\calV$ becomes the $(\calV\times\calV)$-indexed real Hadamard matrix with entries $\bfGamma(v_1,v_2)=(-1)^{\rmB(v_1,v_2)}$.
A bilinear form on $\calV$ is \textit{alternating} if $\rmB(v,v)=0$ for all $v\in\calV$,
and is \textit{symmetric} if $\rmB(v_1,v_2)=\rmB(v_2,v_1)$ for all $v_1,v_2\in\calV$.
Here, every alternating form is symmetric since
$0=\rmB(v_1+v_2,v_1+v_2)=0+\rmB(v_1,v_2)+B(v_2,v_2)+0$.

A \textit{symplectic form} on $\calV$ is a nondegenerate alternating bilinear form on $\calV$.
For such forms,
$\bfGamma(v_1,v_2)=(-1)^{\rmB(v_1,v_2)}$ defines a real symmetric Hadamard matrix whose diagonal entries have value $1$.
In particular, letting $V=\#(\calV)$ we have $\bfGamma^2=\bfGamma^*\bfGamma=V\bfI$ and $\Tr(\bfGamma)=V$,
implying that $\bfGamma$ has eigenvalues $\sqrt{V}$ and $-\sqrt{V}$ with multiplicities $\frac12(V+\sqrt{V})$ and $\frac12(V-\sqrt{V})$, respectively.
In particular, a symplectic form on $\calV$ can only exist if $\sqrt{V}$ is an integer,
namely only if the dimension of $\calV$ over $\bbF_2$ is even.
In such cases, $\sqrt{V}\bfI+\bfGamma$ is the Gram matrix of a real $\ETF(\frac12(V+\sqrt{V}),V)$.
Such \textit{symplectic ETFs} are well known,
with various special cases and generalizations of them and their Naimark complements appearing in the literature in numerous guises,
including Theorem~5.4 of~\cite{BodmannE10},
Theorem~4.1 of~\cite{CoutinkhoGSZ16} (when applied to the \textit{Thas--Somma} construction of \textit{distance-regular antipodal covers of complete graphs} (DRACKNs)),
Theorem~5.1 of~\cite{FickusJMPW19},
Theorem~6.4 of~\cite{IversonJM16},
and Theorem~4.11 of~\cite{BodmannK20};
see Example 6.10 of~\cite{IversonM20} for more discussion.

\begin{example}
\label{ex.symplectic form}
For any positive integer $M$ and  $\bfx=(x_1,\dotsc,x_{2M}),\,\bfy=(y_1,\dotsc,y_{2M})\in\bbF_2^{2M}$ let
\begin{equation}
\label{eq.canonical symplectic}
\rmB(\bfx,\bfy)
:=\sum_{m=1}^M (x_{2m-1}y_{2m}+x_{2m}y_{2m-1})
=(x_1y_2+x_2y_1)+\dotsb+(x_{2M-1}y_{2M}+x_{2M}y_{2M-1}).
\end{equation}
That is, $\rmB(\bfx,\bfy)=\bfx^\rmT\bfB\bfy$ where $\bfB$ is the $2M\times 2M$ block diagonal matrix over $\bbF_2$ whose $M$ diagonal blocks are all $[\begin{smallmatrix}0&1\\1&0\end{smallmatrix}]$.
This is clearly a bilinear form on $\bbF_2^{2M}$.
It is moreover alternating
(since every summand of~\eqref{eq.canonical symplectic} is zero when $\bfx=\bfy$)
and nondegenerate (since if $\bfB(\bfx,\bfy)=0$ for all members $\bfx$ of the standard basis then $\bfy=0$).
It is thus a symplectic form on $\bbF_2^{2M}$.
Provided we order the members of $\bbF_2^{2M}$ lexicographically,
the corresponding character table $\bfGamma$ can be formed by tensoring together $M$ copies of
\begin{equation*}
\left[\begin{smallmatrix}
+&+&+&+\\
+&+&-&-\\
+&-&+&-\\
+&-&-&+
\end{smallmatrix}\right].
\end{equation*}
Taking $M=2$ for instance yields the symplectic form $\rmB$ of Example~\ref{ex.PDS} and the resulting character table $\bfGamma$ of~\eqref{eq.16 x 16 Gamma}.
Though not necessary for our work below,
it is known that up to isomorphism,
this is the only symplectic form on a vector space $\calV$ over $\bbF_2$ of dimension $2M$~\cite{Grove02}.
The binary matrices that preserve the form given in~\eqref{eq.canonical symplectic} form the classical \textit{symplectic group} $\mathrm{Sp}(2M,2)$.
\end{example}

A \textit{quadratic form} on $\calV$ is any function $\rmQ:\calV\rightarrow\bbF_2$ such that
\begin{equation}
\label{eq.polarization}
\rmB(v_1,v_2):=\rmQ(v_1+v_2)+\rmQ(v_1)+\rmQ(v_2)
\end{equation}
defines a bilinear form $\rmB$ on $\calV$.
Any such bilinear form $\rmB$ is necessarily alternating since
$\rmB(v,v)=\rmQ(v+v)+\rmQ(v)+\rmQ(v)=\rmQ(0)$ for any $v\in\calV$ where,
as a special case of this, $\rmQ(0)=\rmB(0,0)=0$.
A point $v\in\calV$ is \textit{singular} with respect to a given quadratic form $\rmQ$ if $\rmQ(v)=0$ and is otherwise \textit{nonsingular}.
The \textit{quadric} of $\rmQ$ is the set $\calD=\set{v\in\calV: \rmQ(v)=0}$ of its singular vectors, which necessarily includes $v=0$.
Remarkably, the quadratic form $\rmQ$ that gives rise to a particular bilinear form $\rmB$ is not unique in general.
For example,
for any $v_0\in\calV$, the function $\tilde{\rmQ}:\calV\rightarrow\bbF_2$, $\tilde{\rmQ}(v):=\rmQ(v+v_0)+\rmQ(v_0)=\rmB(v,v_0)+\rmQ(v)$ also satisfies~\eqref{eq.polarization},
since for any $v_1,v_2\in\calV$,
\begin{align*}
\tilde{\rmQ}(v_1+v_2)+\tilde{\rmQ}(v_1)+\tilde{\rmQ}(v_2)
&=\rmB(v_1+v_2,v_0)+\rmQ(v_1+v_2)+\rmB(v_1,v_0)+\rmQ(v_1)+\rmB(v_2,v_0)+\rmQ(v_2)\\
&=\rmQ(v_1+v_2)+\rmQ(v_1)+\rmQ(v_2).
\end{align*}
The quadric $\tilde{\calD}$ of $\tilde{\rmQ}$ is the corresponding shift of either the quadric $\calD$ of $\rmQ$ or its complement,
depending on whether or not $v_0$ is singular:
\begin{equation*}
\tilde{\calD}
=\set{v\in\calV: \tilde{\rmQ}(v)=0}
=\set{v\in\calV: \rmQ(v+v_0)=\rmQ(v_0)}
=\left\{\begin{array}{ll}
v_0+\calD,     &\ v_0\in\calD,\\
v_0+\calD^\rmc,&\ v_0\in\calD^\rmc.
\end{array}\right.
\end{equation*}

We focus on quadratic forms $\rmQ:\calV\rightarrow\bbF_2$ that are \textit{nondefective}, that is, for which the bilinear form~\eqref{eq.polarization} is nondegenerate, and so symplectic.
Here, $V=\#(\calV)=2^{2M}$ for some positive integer $M$.
Such forms are classical, and we now summarize and explain the parts of their folklore that we shall make use of:

\begin{lemma}
\label{lem.chirp facts}
Let $\calV$ be a vector space over $\bbF_2$ of cardinality $2^{2M}$.
Let $\rmQ:\calV\rightarrow\bbF_2$ be a nondefective quadratic form with associated symplectic form~\eqref{eq.polarization} and sign
\begin{equation}
\label{eq.sign of quadratic form}
\sgn(\rmQ):=\frac1{2^M}\sum_{v\in\calV}(-1)^{\rmQ(v)}\in\set{1,-1}.
\end{equation}
(If the opposite sign is desired, replace $\rmQ$ with $\tilde{\rmQ}(v):=\rmQ(v+v_0)+1$ where $Q(v_0)=1$.)
Then the $(\calV\times\calV)$-indexed matrices $\bfGamma$, $\bfC$ and $\bfDelta$ defined by
\begin{equation*}
\bfGamma(v_1,v_2)=(-1)^{\rmB(v_1,v_2)},
\quad
\bfC(v_1,v_2)=(-1)^{\rmQ(v_1+v_2)},
\quad
\bfDelta(v_1,v_2)
=\left\{\begin{array}{cl}
(-1)^{\rmQ(v_1)},&\ v_1=v_2,\\
0,&\ v_1\neq v_2,
\end{array}\right.
\end{equation*}
are real-symmetric, and satisfy
\begin{equation}
\label{eq.chirp Fourier}
\bfGamma^2=2^{2M}\bfI=\bfC^2,
\quad
\bfDelta^2=\bfI,
\quad
\bfGamma=\bfDelta\bfC\bfDelta,
\quad
(\bfGamma\bfDelta)^3=2^{3M}\sgn(\rmQ)\bfI.
\end{equation}
Moreover, $\calD=\set{v\in\calV: \rmQ(v)=0}$ is a difference set for $\calV$ of cardinality $2^{M-1}[2^M+\sgn(\rmQ)]$ whose harmonic ETF
$\set{\bfphi_v}_{v\in\calV}$,
\smash{$\bfphi_v(d):=(-1)^{\rmB(d,v)}$} has Gram matrix $2^{M-1}[2^M\bfI+\sgn(\rmQ)\bfC]$.
\end{lemma}

\begin{proof}
Consider the \textit{(quadratic) chirp} function $\bfc\in\bbR^{\calV}$, $\bfc(v):=(-1)^{\rmQ(v)}$.
This is an eigenvector of the DFT matrix $\bfGamma^*=\bfGamma$: for any $v_1\in\calV$,
using~\eqref{eq.polarization} and making the substitution $v=v_1+v_2$ gives
\begin{equation}
\label{eq.chirp is eigenvector}
(\bfGamma\bfc)(v_1)
=\sum_{v_2\in\calV}(-1)^{\rmB(v_1,v_2)}(-1)^{\rmQ(v_2)}
=\sum_{v_2\in\calV}(-1)^{\rmQ(v_1+v_2)+\rmQ(v_1)}
=\biggparen{\,\sum_{v\in\calV}(-1)^{\rmQ(v)}}\bfc(v_1).
\end{equation}
Hence the ``Gauss sum" $\sum_{v\in\calV}(-1)^{\rmQ(v)}$ is an eigenvalue of $\bfGamma$, and so is either $\sqrt{V}=2^M$ or $-\sqrt{V}=-2^M$.
In particular, the \textit{sign}~\eqref{eq.sign of quadratic form} of $\rmQ$ is indeed either $1$ or $-1$.
We caution that the sign of a quadratic form $\rmQ$ that gives rise to a particular symplectic form $\rmB$ is determined by $\rmQ$ but not by $\rmB$: for example, $\sgn(\tilde{\rmQ})=(-1)^{\rmQ(v_0)}\sgn(\rmQ)$ where $\tilde{\rmQ}(v):=\rmQ(v+v_0)+\rmQ(v_0)$.
A nondefective quadratic form $\rmQ$ is called \textit{hyperbolic} when $\sgn(\rmQ)=1$ and called \textit{elliptic} when $\sgn(\rmQ)=-1$.

Under this notation, \eqref{eq.chirp is eigenvector} becomes $\bfGamma\bfc=2^M\,\sgn(\rmQ)\bfc$.
Since $\bfc=2\bfchi_\calD-\bfone$ where $\calD$ is the quadric of $\rmQ$,
this can be restated as
$2\bfGamma\bfchi_\calD-2^{2M}\bfdelta_0
=\bfGamma(2\bfchi_\calD-\bfone)
=2^M\,\sgn(\rmQ)(2\bfchi_\calD-\bfone)$,
that is,
\begin{equation}
\label{eq.DFT derivation}
\bfGamma\bfchi_\calD
=2^{M-1}[2^M\bfdelta_0+\sgn(\rmQ)\bfc]
=2^{M-1}[2^M\bfdelta_0+\sgn(\rmQ)(2\bfchi_\calD-\bfone)].
\end{equation}
When evaluated at $v=0$, this gives
$\#(\calD)=(\bfGamma\bfchi_\calD)(0)=2^{M-1}[2^M+\sgn(\rmQ)]$.
When instead evaluated at any $v\neq 0$, \eqref{eq.DFT derivation} gives
$\abs{(\bfGamma\bfchi_\calD)(v)}=2^{M-1}$.
As such, $\calD$ is a difference set for $\calV$.
(From this perspective, \eqref{eq.DFT derivation} itself is remarkable: for many difference sets $\calD$,
no simple expression for the phase of the DFT of $\bfchi_\calD$ is known.)
In particular, \eqref{eq.DFT derivation} implies that the corresponding harmonic ETF $\set{\bfphi_v}_{v\in\calV}$ for $\bbR^\calD$,
\smash{$\bfphi_v(d):=(-1)^{\rmB(d,v)}$}, satisfies
\begin{equation*}
\ip{\bfphi_{v_1}}{\bfphi_{v_2}}
=\sum_{d\in\calD}(-1)^{\rmB(d,v_1+v_2)}
=(\bfGamma\bfchi_\calD)(v_1+v_2)
=2^{M-1}\left\{\begin{array}{cl}
2^M+\sgn(\rmQ),&\ v_1=v_2,\\
\sgn(\rmQ)\bfc(v_1+v_2),&\ v_1\neq v_2.
\end{array}\right.
\end{equation*}
That is, $\set{\bfphi_v}_{v\in\calV}$ has Gram matrix $\bfPhi^*\bfPhi=2^{M-1}[2^M\bfI+\sgn(\rmQ)\bfC]$,
where $\bfC$ is the ($\calV$-circulant) filter defined by $\bfC\bfx:=\bfc*\bfx$,
namely where $\bfC(v_1,v_2)=\bfc(v_1+v_2)=(-1)^{\rmQ(v_1+v_2)}$.
Like all filters over $\calV$, this matrix $\bfC$ is diagonalized by the DFT $\bfGamma$:
for any $\bfx\in\bbR^\calV$, $v\in\calV$,
\begin{equation*}
(\bfGamma\bfC\bfx)(v)
=[\bfGamma(\bfc*\bfx)](v)
=(\bfGamma\bfc)(v)(\bfGamma\bfx)(v)
=2^M\,\sgn(\rmQ)\bfc(v)(\bfGamma\bfx)(v)
=2^M\,\sgn(\rmQ)(\bfDelta\bfGamma\bfx)(v),
\end{equation*}
where $\bfDelta:\bbR^\calV\rightarrow\bbR^\calV$,
$(\bfDelta\bfx)(v):=\bfc(v)\bfx(v)$ is the \textit{chirp modulation} operator,
namely the diagonal $(\calV\times\calV)$-indexed orthogonal matrix $\bfDelta$ whose $v$th diagonal entry is $\bfDelta(v,v)=\bfc(v)=(-1)^{\rmQ(v)}$.
That is, $\bfC=2^{-M}\sgn(\rmQ)\bfGamma\bfDelta\bfGamma$.
It is remarkable however that $\bfGamma$ and $\bfC$ are also related by conjugation by $\bfDelta$:
for any $v_1,v_2\in\calV$, \eqref{eq.polarization} gives
\begin{equation*}
(\bfDelta\bfC\bfDelta)(v_1,v_2)
=(-1)^{\rmQ(v_1)}(-1)^{\rmQ(v_1+v_2)}(-1)^{\rmQ(v_2)}
=(-1)^{\rmB(v_1,v_2)}
=\bfGamma(v_1,v_2),
\end{equation*}
and so $\bfGamma=\bfDelta\bfC\bfDelta$.
In particular, $\bfC=\bfDelta\bfGamma\bfDelta$ (like $\bfGamma$) is a real-symmetric Hadamard matrix whose diagonal entries have value $1$.
(We caution that $\bfGamma$ and $\bfC$ are distinct:
$\bfC$ is $\calV$-circulant whereas $\bfGamma$ is not, with the latter having an all-ones first column.)
Moreover, combining the above facts gives
\smash{$\bfDelta\bfGamma\bfDelta=\bfC=2^{-M}\sgn(\rmQ)\bfGamma\bfDelta\bfGamma$},
namely that the \textit{Fourier-chirp} transform $\bfGamma\bfDelta$ satisfies $(\bfGamma\bfDelta)^3=2^{3M}\sgn(\rmQ)\bfI$.
Analogous transforms arise in the study of \textit{SIC-POVMs}; see~Section~3.4 of~\cite{Zauner99}, and~\cite{Fickus09}.
Interestingly, this implies that the shifts of $\bfc$ form an equal-norm orthogonal basis of eigenvectors of the DFT,
that is, $\bfGamma$ and $\bfC$ orthogonally diagonalize each other:
\begin{equation*}
\bfGamma
=2^{-3M}\sgn(\rmQ)\bfGamma(\bfGamma\bfDelta)^3
=2^{-M}\sgn(\rmQ)(\bfDelta\bfGamma\bfDelta)\bfDelta(\bfDelta\bfGamma\bfDelta)
=2^{-M}\sgn(\rmQ)\bfC\bfDelta\bfC.\qedhere
\end{equation*}
\end{proof}

\begin{example}
\label{ex.quadratic forms}
Continuing Example~\ref{ex.symplectic form}, for any positive integer $M$,
the bilinear form $\rmB$ on $\calV=\bbF_2^M$ given in~\eqref{eq.canonical symplectic} arises via~\eqref{eq.polarization}, for example, from the quadratic form $\rmQ:\bbF_2^{2M}\rightarrow\bbF_2$,
\begin{equation}
\label{eq.canonical hyperbolic}
\rmQ(\bfx)
=\rmQ(x_1,\dotsc,x_{2M})
:=\sum_{m=1}^M x_{2m-1}x_{2m}
=x_1x_2+\dotsc+x_{2M-1}x_{2M}.
\end{equation}
Since $\rmB$ is nondegenerate (symplectic),
$\rmQ$ is nondefective.
When $M=1$, $\rmQ(x_1,x_2)=x_1x_2$ has three singular vectors and one nonsingular one: its quadric is $\calD=\set{00,01,10}$, and so $\calD^\rmc=\set{11}$.
For $M>1$, a vector in $\bbF_2^{2M}$ is singular if and only if it is obtained by either appending $00$, $10$ or $01$ to a singular vector in $\bbF_2^{\smash{2M-2}}$ or appending $11$ to a nonsingular one.
By induction, this implies that~\eqref{eq.canonical hyperbolic} has exactly
$\#(\calD)=2^{M-1}(2^M+1)$ singular vectors, and so is hyperbolic.
For an elliptic quadratic form $\tilde{\rmQ}$ that yields the same symplectic form $\rmB$, we can for example let $\tilde{\rmQ}(\bfx):=\rmQ(\bfx+\bfx_0)+1$ where $\rmQ(\bfx_0)=1$.
Taking $\bfx_0$ to be $00\dotsb0011$ for instance yields
\begin{equation}
\label{eq.canonical elliptic}
\tilde{\rmQ}(\bfx)
:=\sum_{m=1}^M x_{2m-1}x_{2m}+x_{2M-1}^2+x_{2M}^2
=x_1x_2+\dotsc+x_{2M-1}x_{2M}+x_{2M-1}^2+x_{2M}^2.
\end{equation}
When $M=2$ this becomes the elliptic quadratic form $x_1x_2+x_3x_4+x_3^2+x_4^2$ used in Example~\ref{ex.PDS} whose $6$-element quadric $\calD$ is given in~\eqref{eq.PDS(16,6,10)}.
(Adding $0011$ to these vectors gives the nonsingular vectors of the hyperbolic quadratic form $x_1x_2+x_3x_4$.)
The corresponding chirp $\bfc$ has values
$+---+---+----+++$.
Conjugating $\bfGamma$ of~\eqref{eq.16 x 16 Gamma} by the diagonal matrix $\bfDelta$ gives $\bfC=\bfDelta\bfGamma\bfDelta$ (the unique $\bbF_2^{2M}$-circulant matrix that has $\bfc$ as its first column).
This matrix $\bfC$ naturally arises in the Gram matrix $\bfGamma_0^*\bfGamma_0^{}=2(4\bfI-\bfC)=2\bfDelta(4\bfI-\bfGamma)\bfDelta$ of the corresponding harmonic $\ETF(6,16)$ whose synthesis operator $\bfGamma_0$ is given in~\eqref{eq.16 x 16 Gamma}.
Here, $\bfGamma\bfDelta$ is a real Hadamard matrix with the remarkable property that $(\bfGamma\bfDelta)^3=-64\bfI$.
In the next result we use this Fourier-chirp relation to prove that $\calD$ and $\calD^\rmc$ are paired difference sets.

Though not necessary for our work below,
it is known that up to isomorphism,
\eqref{eq.canonical hyperbolic} and \eqref{eq.canonical elliptic} are the only hyperbolic and elliptic nondefective quadratic forms on a vector space $\calV$ over $\bbF_2$ of dimension $2M$~\cite{Grove02}.
The binary matrices that preserve the form given in~\eqref{eq.canonical hyperbolic} or~\eqref{eq.canonical elliptic} form the classical \textit{orthogonal groups} $\mathrm{O}^+(2M,2)$ and $\mathrm{O}^{-}(2M,2)$, respectively.
\end{example}

\begin{theorem}
\label{thm.PDS from quadratic}
Let $\rmQ$ be a nondefective quadratic form on a vector space $\calV$ over $\bbF_2$,
and let $\rmB$ be the associated symplectic form~\eqref{eq.polarization}.
Then the set $\calD=\set{v\in\calV: \rmQ(v)=0}$ of all singular vectors of $\rmQ$ is a difference set for $\calV$ that is paired with $\calD^\rmc$ in the sense of Definition~\ref{def.paired difference sets}, provided we identify (the additive group of) $\calV$ with its Pontryagin dual via the isomorphism $v_2\mapsto(v_1\mapsto(-1)^{\rmB(v_1,v_2)})$.
\end{theorem}

\begin{proof}
Recall the notation and facts of Lemma~\ref{lem.chirp facts}.
We already know that $\calD$ is a difference set for $\calV$
(and so $\calD^\rmc$ is as well),
and that the synthesis operator $\bfGamma_0$ of the resulting harmonic ETF
(the $(\calD\times\calV)$-indexed submatrix of $\bfGamma$) satisfies $\bfGamma_0^*\bfGamma_0^{}=2^{M-1}[2^M\bfI+\sgn(\rmQ)\bfC]$.
To show that $\calD$ and $\calD^\rmc$ are paired in the sense of Definition~\ref{def.paired difference sets},
we want to show that $\calD^\rmc$-indexed columns of $\bfGamma_0$ form a tight frame for their span, or equivalently, that the $(\calD^\rmc\times\calD^\rmc)$-indexed submatrix of $\bfGamma_0^*\bfGamma_0^{}=2^{M-1}[2^M\bfI+\sgn(\rmQ)\bfC]$ is a scalar multiple of a projection.
Here since $\frac12(\bfI-\bfDelta)$ is the diagonal $\set{0,1}$-valued matrix whose diagonal entries indicate $\calD^\rmc$,
this equates to showing that
\begin{equation*}
\bfG
:=\tfrac12(\bfI-\bfDelta)[2^M\bfI+\sgn(\rmQ)\bfC]\tfrac12(\bfI-\bfDelta)
\end{equation*}
satisfies $\bfG^2=A\bfG$ for some $A>0$.
To simplify this expression for $\bfG$,
note that since $\bfDelta\bfC\bfDelta=\bfGamma$ where $\bfDelta^2=\bfI$
(and so $[\frac12(\bfI-\bfDelta)]^2=\frac12(\bfI-\bfDelta)$ and $(\bfI-\bfDelta)\bfDelta=-(\bfI-\bfDelta)$),
\begin{equation}
\label{eq.pf of infinite family 1}
\bfG
=\tfrac12(\bfI-\bfDelta)[2^M\bfI+\sgn(\rmQ)\bfDelta\bfGamma\bfDelta]\tfrac12(\bfI-\bfDelta)
=2^M\tfrac{1}{2}(\bfI-\bfDelta)+\sgn(\rmQ)\tfrac12(\bfI-\bfDelta)\bfGamma\tfrac12(\bfI-\bfDelta).
\end{equation}
Squaring this equation and again making use of the fact that
$[\frac12(\bfI-\bfDelta)]^2=\frac12(\bfI-\bfDelta)$ gives
\begin{equation}
\label{eq.pf of infinite family 2}
\bfG^2
=2^{2M}\tfrac12(\bfI-\bfDelta)+2^{M+1}\,\sgn(\rmQ)\tfrac12(\bfI-\bfDelta)\bfGamma\tfrac12(\bfI-\bfDelta)
+\tfrac12(\bfI-\bfDelta)\bfGamma\tfrac12(\bfI-\bfDelta)\bfGamma\tfrac12(\bfI-\bfDelta).
\end{equation}
To proceed, recall from~\eqref{eq.chirp Fourier} that
$(\bfGamma\bfDelta\bfGamma)(\bfDelta\bfGamma\bfDelta)
=(\bfGamma\bfDelta)^3=2^{3M}\sgn(\rmQ)\bfI$
where $\bfGamma^2=2^{2M}\bfI$ and $\bfDelta^2=\bfI$.
Thus, $\bfGamma\bfDelta\bfGamma=2^M\,\sgn(\rmQ)\bfDelta\bfGamma\bfDelta$ and so
\begin{equation*}
\bfGamma\tfrac12(\bfI-\bfDelta)\bfGamma
=2^{2M-1}\bfI-\tfrac12\bfGamma\bfDelta\bfGamma
=2^{2M-1}\bfI-2^{M-1}\,\sgn(\rmQ)\bfDelta\bfGamma\bfDelta.
\end{equation*}
Conjugating this equation by $\tfrac12(\bfI-\bfDelta)$ gives
\begin{equation*}
\tfrac12(\bfI-\bfDelta)\bfGamma\tfrac12(\bfI-\bfDelta)\bfGamma\tfrac12(\bfI-\bfDelta)
=2^{2M-1}\tfrac12(\bfI-\bfDelta)-2^{M-1}\,\sgn(\rmQ)\tfrac12(\bfI-\bfDelta)\bfGamma\tfrac12(\bfI-\bfDelta).
\end{equation*}
Substituting this into~\eqref{eq.pf of infinite family 2} and then recalling~\eqref{eq.pf of infinite family 1} gives that $\bfG^2=A\bfG$ for some $A>0$:
\begin{align*}
\bfG^2
&=2^{2M}\tfrac12(\bfI-\bfDelta)
+2^{M+1}\,\sgn(\rmQ)\tfrac12(\bfI-\bfDelta)\bfGamma\tfrac12(\bfI-\bfDelta)\\
&\qquad+2^{2M-1}\tfrac12(\bfI-\bfDelta)-2^{M-1}\,\sgn(\rmQ)\tfrac12(\bfI-\bfDelta)\bfGamma\tfrac12(\bfI-\bfDelta)\\
&=3(2^{M-1})[2^M\tfrac12(\bfI-\bfDelta)
+\sgn(\rmQ)\tfrac12(\bfI-\bfDelta)\bfGamma\tfrac12(\bfI-\bfDelta)]
=3(2^{M-1})\bfG.\qedhere
\end{align*}
\end{proof}

Applying Theorems~\ref{thm.ECTFF from PDS} and~\ref{thm.PDS from quadratic} to the canonical quadratic forms of Example~\ref{ex.quadratic forms} yields the following result:

\begin{theorem}
\label{thm.infinite family}
For any positive integer $M$,
let $\rmQ$ be either the hyperbolic~\eqref{eq.canonical hyperbolic} or elliptic~\eqref{eq.canonical elliptic} quadratic form over $\bbF_2^{2M}$ with associated symplectic form $\rmB$ of~\eqref{eq.canonical symplectic}
(which have $\sgn(\rmQ)=1$ and $\sgn(\rmQ)=-1$, respectively)
and let
$\calD=\set{\bfx\in\bbF_2^{2M}: \rmQ(\bfx)=0}$.
Then $\calD$ and $\calD^\rmc$ are paired difference sets for $\bbF_2^{2M}$
(Theorem~\ref{thm.PDS from quadratic}),
and applying Theorem~\ref{thm.ECTFF from PDS} to them gives that:
\begin{enumerate}
\renewcommand{\labelenumi}{(\alph{enumi})}
\item
$\set{\bfphi_\bfy}_{\bfy\in\bbF_2^M}\subseteq\bbR^\calD$, $\bfphi_{\bfy}(\bfx)=(-1)^{\rmB(\bfx,\bfy)}$ is an $\ETF(2^{M-1}[2^M+\sgn(\rmQ)],2^{2M})$ for $\bbR^\calD$;\smallskip
\item
for any $\bfy\in\bbF_2^{2M}$, $\set{\bfphi_{\bfy+\bfz}}_{\bfz\in\calD^\rmc}$ is an $\ETF(\frac13(2^{2M}-1),2^{M-1}[2^M-\sgn(\rmQ)])$ for its span $\calU_\bfy$;\smallskip
\item
$\set{\calU_\bfy}_{\bfy\in\bbF_2^{2M}}$ is an
$\ECTFF(2^{M-1}[2^M+\sgn(\rmQ)],2^{2M},\frac13(2^{2M}-1))$ for $\bbR^\calD$;\smallskip
\item
$\set{\bfpsi_\bfy}_{\bfy\in\bbF_2^M}\subseteq\bbR^{\calD^\rmc}$, $\bfpsi_{\bfy}(\bfx)=(-1)^{\rmB(\bfx,\bfy)}$ is an $\ETF(2^{M-1}[2^M-\sgn(\rmQ)],2^{2M})$ for $\bbR^{\calD^\rmc}$;\smallskip
\item
for any $\bfy\in\bbF_2^{2M}$, $\set{\bfpsi_{\bfy+\bfz}}_{\bfz\in\calD}$ is an $\ETF(\frac13(2^{2M}-1),2^{M-1}[2^M+\sgn(\rmQ)])$ for its span $\calV_\bfy$;\smallskip
\item
$\set{\calV_\bfy}_{\bfy\in\bbF_2^{2M}}$ is an
$\ECTFF(2^{M-1}[2^M-\sgn(\rmQ)],2^{2M},\frac13(2^{2M}-1))$ for $\bbR^{\calD^\rmc}$.
\end{enumerate}
\end{theorem}

\begin{proof}
Most of these results are immediate consequences of Theorems~\ref{thm.ECTFF from PDS} and~\ref{thm.PDS from quadratic}.
To find the dimension $R$ of the spans of the sub-ETFs in (b) and (e) we use~\eqref{eq.paired diff set rank} where
$N=\#(\calV)=2^{2M}$,
$D=\#(\calD)=2^{M-1}[2^M+\sgn(\rmQ)]$
and
$E=\#(\calD^\rmc)=2^{2M}-2^{M-1}[2^M+\sgn(\rmQ)]=2^{M-1}[2^M-\sgn(\rmQ)]$.
Here, since $DE=2^{2M-2}(N-1)$ and $D+E=N$,
\begin{equation*}
R
=\tfrac{DE(N-1)}{(D+E-1)N-DE}
=\tfrac{2^{2M-2}(N-1)^2}{(N-1)N-2^{2M-2}(N-1)}
=\tfrac{2^{2M-2}(N-1)}{N-2^{2M-2}}
=\tfrac{2^{2M-2}(2^{2M}-1)}{2^{2M}-2^{2M-2}}
=\tfrac13(2^{2M}-1).
\end{equation*}
Another peculiarity of these examples is that since $\bfGamma$ is symmetric and $\calD$ and $\calD^\rmc$ are complementary,
it is valid to construct the ECTFFs (c) and (f) that arise from Theorem~\ref{thm.ECTFF from PDS} in this setting from sub-ETFs (b) and (e) of Naimark complementary ETFs (a) and (d), respectively.
\end{proof}

As discussed in Example~\ref{ex.ECTFF}, when $M=2$ this yields both an $\ETF(6,16)$ that contains $16$ (distinct, but overlapping and unitarily equivalent) sub-$\ETF(5,10)$ whose spans form an $\ECTFF(6,16,5)$,
as well as an $\ETF(10,16)$ that contains $16$ sub-$\ETF(5,6)$ whose spans form an $\ECTFF(10,16,5)$.
When instead $M=3$, it yields an $\ETF(28,64)$ that contains $64$ sub-$\ETF(21,36)$ whose spans form an $\ECTFF(28,64,21)$ as well as an $\ETF(36,64)$ that contains $64$ sub-$\ETF(21,28)$ whose spans form an $\ECTFF(36,64,21)$.
These alone account for a remarkable proportion of real ETFs with small parameters~\cite{FickusM16}.
(When $M=1$, Theorem~\ref{thm.infinite family} becomes trivial,
yielding an $\ETF(1,4)$ that contains $4$ sub-$\ETF(1,3)$ whose spans form an $\ECTFF(1,4,1)$ and an $\ETF(3,4)$ that contains $4$ sub-$\ETF(1,1)$ whose spans form an $\ECTFF(3,4,1)$.)

With the exception of this $\ECTFF(6,16,5)$ (which as already noted is the spatial complement of an $\ETF(6,16)$), the ECTFFs produced by Theorem~\ref{thm.infinite family} when $M\geq 2$ seem to be new:
we could not find any other way to construct (real or complex) ECTFFs with these parameters from any of the methods mentioned in the introduction.
These ECTFFs cannot be EITFFs: as noted in the previous section, this is actually true of any ECTFFs that arise from nontrivial paired difference sets since the resulting subspaces intersect nontrivially;
here, this also follows from the fact that $2R>D$.
A more interesting question is whether the spatial complements of these ECTFFs are EITFFs.
When $M=2$, the spatial complement of the $\ECTFF(6,16,5)$ certainly is an EITFF, while that of the $\ECTFF(10,16,5)$ certainly is not (since having $D=2R$ implies that its principal angles do not change under spatial complements).
Our preliminary numerical experimentation indicates that when $M\geq 3$ the spatial complements of the ECTFFs of Theorem~\ref{thm.infinite family} are not EITFFs in general.
When $M\geq 3$ an ECTFF of (c) or (f) has $D<2R$ and so any pair of its subspaces have at least
$2R-D
=\frac23(2^M\pm1)(2^{M-2}\mp 1)$ principal angles of $0$;
only when every pair of its subspaces has exactly $2R-D$ principal angles of $0$ and $D-R$ principal angles of some other constant value will its spatial complement be an EITFF.

\begin{remark}
Every ETF constructed in Theorem~\ref{thm.infinite family} equates to a known type of strongly regular graph (SRG)~\cite{Brouwer07,Brouwer17}.
A graph on a $V$-element vertex set $\calV$ is \textit{strongly regular} with parameters $(V,K,\Lambda,U)$ if its adjacency matrix $\bfA$ satisfies
$\bfA^2=(\Lambda-U)\bfA+(K-U)\bfI+U\bfJ$.
These parameters are dependent:
since $\bfA\bfone=K\bfone$, applying $\bfA^2$ to $\bfone$ gives $K^2=(\Lambda-U)K+(K-U)+UV$, namely that $U(V-K-1)=K(K-\Lambda-1)$.
In general, there are two notions of equivalence between certain real ETFs and certain SRGs.
By negating some vectors if necessary,
every real $N$-vector ETF is \textit{projectively equivalent} to one $\set{\bfphi_n}_{n\in\calN}$ for which there exists $n_0\in\calN$ such that $\ip{\bfphi_{n_0}}{\bfphi_{n}}>0$ for all $n$.
Such an ETF equates to an SRG on the vertex set $\calN\backslash\set{n_0}$ with $K=2U$~\cite{HolmesP04,Waldron09}.
Here, two vertices $n_1,n_2\in\calN\backslash\set{n_0}$ are adjacent when $\ip{\bfphi_{n_1}}{\bfphi_{n_2}}>0$~\cite{FickusJMPW18} and
\begin{equation}
\label{eq.traditional ETF SRG equivalence}
V=N-1,
\quad
K=\tfrac N2-1-(\tfrac N{2D}-1)\bigbracket{\tfrac{D(N-1)}{N-D}}^{\frac12},
\quad
U=\tfrac K2.
\end{equation}
(In~\cite{Waldron09},
adjacency instead equates to having $\ip{\bfphi_{n_1}}{\bfphi_{n_2}}<0$,
yielding the complementary graph.)
Sometimes a real $N$-vector ETF $\set{\bfphi_n}_{n\in\calN}$ for some Hilbert space $\bbH$ instead has the all-ones vector $\bfone$ as an eigenvector of its Gram matrix.
This can occur in two distinct ways: either the ETF is \textit{centered}, having $\bfone\in\ker(\bfPhi)=\ker(\bfPhi^*\bfPhi)$ and so $\sum_{n\in\calN}\bfphi_n=\bfPhi\bfone=\bfzero$,
or is \textit{axial}, having $\bfone\in\bfPhi^*\bfPhi(\bbR^\calN)=\bfPhi^*(\bbH)$,
meaning all of its vectors make the same angle with their nonzero centroid.
An axial real $\ETF(D,N)$ $\set{\bfphi_n}_{n\in\calN}$ equates~\cite{FickusJMPW18} to an SRG on the vertex set $\calN$ with $V=4K-2\Lambda-2U$ and $V-2K-1<0$.
Here, two vertices $n_1,n_2\in\calN$ are adjacent if and only if $\ip{\bfphi_{n_1}}{\bfphi_{n_2}}>0$ and
\begin{equation}
\label{eq.axial ETF SRG equivalence}
V=N,
\quad
K=\tfrac{N-1}{2}+\tfrac12(\tfrac{N}{D}-1)\bigbracket{\tfrac{D(N-1)}{N-D}}^{\frac12},
\quad
U=\tfrac K2\tfrac{V-2K-2}{V-2K-1}.
\end{equation}
(An analogous characterization of centered real ETFs is also known but is superfluous since an ETF is axial if and only if its Naimark complement is centered~\cite{FickusJMPW18}.)
While every real ETF arises from the equivalence of~\eqref{eq.traditional ETF SRG equivalence}, it is an open problem if the same holds for~\eqref{eq.axial ETF SRG equivalence}: we do not know if every real ETF with $V$ vectors is projectively equivalent to one whose signature matrix matches the Seidel adjacency matrix of a $(V,K,\Lambda,U)$-SRG with $V=4K-2\Lambda-2U$~\cite{FickusJMPW18}.

With respect to the ETFs of Theorem~\ref{thm.infinite family},
note that since the zero vector in $\bbF_2^{2M}$ is singular,
the synthesis operator $\bfGamma_0$ of the harmonic ETF of (a) includes the all-ones ($0$-indexed) row of $\bfGamma$.
See~\eqref{eq.16 x 16 Gamma} for when $M=2$ and $\sgn(\rmQ)=-1$, for example.
As such, this ETF is axial.
Applying~\eqref{eq.axial ETF SRG equivalence} to it yields an SRG with $V=2^{2M}$, $K=\tfrac12[2^M-\sgn(\rmQ)][2^M+\sgn(\rmQ)+1]$ and
$U=2^{M-1}(2^{M-1}+1)$ in which,
since $\bfGamma_0^*\bfGamma_0^{}=2^{M-1}[2^M\bfI+\sgn(\rmQ)\bfC]$,
adjacency depends on the value of  $\rmC(\bfy_1,\bfy_2)=\rmc(\bfy_1+\bfy_2)=(-1)^{\rmQ(\bfy_1+\bfy_2)}$.
This is a known~\textit{affine polar graph} ``$VO^{\pm}_{2M}(2)$"~\cite{Brouwer07}.
Its Naimark complement (the ETF of (d)) equates to the (graph) complement of this SRG.
Since $\bfC=\bfDelta\bfGamma\bfDelta$ the ETF of (a) is moreover projectively equivalent to one with Gram matrix $2^{M-1}[2^M\bfI+\sgn(\rmQ)\bfGamma]$.
Since the entries in the $0$th row and column of $\bfGamma$ are constant,
we can apply~\eqref{eq.axial ETF SRG equivalence} to this ETF to obtain a subordinate SRG on the $V=2^{2M}-1$ vertices of $\bbF_2^{2M}\backslash\set{0}$ in which adjacency depends on the value of $\bfGamma(\bfy_1,\bfy_2)=(-1)^{\rmB(\bfy_1,\bfy_2)}$.
This is a known \textit{symplectic graph} ``$Sp_{2M}(2)$"~\cite{Brouwer07}.

Every ETF of (b) is a sub-ETF of the axial ETF (a) and so is also axial:
if the analysis operator of a given sequence of vectors contains an all-ones vector, then the same is true for any of its subsequences.
The row space of~\eqref{eq.ETF(5,10)}, for example, clearly contains the all-ones vector.
Applying~\eqref{eq.axial ETF SRG equivalence} to it yields an SRG with  $V=2^{M-1}[2^M-\sgn(\rmQ)]$, $K=\frac12(2^{M-1}+1)[2^M-1-\sgn(\rmQ)]$ and
$U=2^{M-3}[2^M+3-\sgn(\rmQ)]$ on the nonsingular points of $\bbF_2^{2M}$ in which adjacency depends on the value of $\rmQ(\bfy_1+\bfy_2)=\rmB(\bfy_1,\bfy_2)$.
This is a known ``$NO^{\pm}_{2M}(2)$" SRG~\cite{Brouwer07}.
We caution that a sub-ETF of a centered ETF is not necessarily centered.
In particular, the Gram matrix of an ETF of (e) is the $(\calD\times\calD)$-indexed submatrix of
$2^{M-1}[2^M\bfI-\sgn(\rmQ)\bfGamma]$ and so is neither axial nor centered, having a ($0$-indexed) row and column with entries of constant value.
Applying~\eqref{eq.traditional ETF SRG equivalence} to it yields an SRG on the $V=2^{M-1}[2^M-\sgn(\rmQ)]-1$ nonzero singular points of $\bbF_2^M$ in which adjacency depends on the value of $\rmQ(\bfy_1+\bfy_2)=\rmB(\bfy_1,\bfy_2)$.
This is a known ``$O_{2M}^{\pm}(2)$" SRG.
In fact, a careful analysis reveals that an ETF of (e) is projectively equivalent to one of (b) with opposite sign (since the quadrics of~\eqref{eq.canonical hyperbolic} and~\eqref{eq.canonical elliptic} are shifts of each other),
meaning that ``$O_{2M}^{\pm}(2)$" is subordinate to $NO_{2M}^{\mp}(2)$.

Real ETFs with the same (or Naimark complementary) parameters as those of (a) and (c) arise from McFarland difference sets~\cite{DingF07} and Steiner ETFs~\cite{GoethalsS70,FickusMT12} from $\BIBD(2^M,2,1)$.
Real ETFs with the same (or Naimark complementary) parameters as those of (b) and (e) arise from Steiner and Tremain~\cite{FickusJMP18} ETFs from $\BIBD(2^M-1,3,1)$.
Whether or not such ETFs are truly equivalent (up to unitary transformations on their spans and signed permutations of their vectors) is a question we leave for future research.
\end{remark}

\section{Conclusions and future work}

We have seen that paired difference sets (Definition~\ref{def.paired difference sets}) yield ECTFFs (Theorem~\ref{thm.ECTFF from PDS}) and that an infinite family of nontrivial such pairs exists (Theorem~\ref{thm.infinite family}).
As noted in~\cite{FickusMJ16}, at least one other nontrivial pair exists.
Like that of Example~\ref{ex.PDS} (the $M=2$ case of Theorem~\ref{thm.infinite family}) it consists of a $6$- and $10$-element subset of a group of order $16$.
But unlike that example, the group in question is $\bbZ_4^2$, not $\bbZ_2^4$.
(Interestingly, paired difference sets in $\bbZ_2^2\times\bbZ_4$ and $\bbZ_2\times\bbZ_8$ do not seem to exist, despite the fact that they too contain difference sets of order $6$ and $10$~\cite{FickusMJ16}.)

Nontrivial paired difference sets seem rare, in general:
our numerical search found only $27$ integer triples $(D,E,N)$ that meet even the simplest conditions on the existence of nontrivial paired difference sets of cardinality $D$ and $E$ (ordered without loss of generality according to size) in some abelian group of order $N$ which is at most $1024$,
namely that $1<D\leq E<N\leq 1024$ and that $\frac{D(D-1)}{N-1}$, $\frac{E(E-1)}{N}$ and $R$ of~\eqref{eq.paired diff set rank} are integers.
These include only four triples such that $D+E\neq N$.
Remarkably, these four triples along with seven others are ruled out by cross-referencing against a table of known difference sets~\cite{Gordon19} that makes use of more sophisticated known necessary conditions.
This itself raises an interesting open problem:
do the cardinalities of any nontrivial paired difference sets always sum to the cardinality of the corresponding group?
Of the $16$ triples that remain, four correspond to those produced by Theorem~\ref{thm.infinite family} when $M=2,3,4,5$,
namely those with $(R,D,E,N)$ parameters $(5,6,10,16)$ (Example~\ref{ex.PDS}),
$(21,28,36,64)$, $(85,120,136,256)$ and $(341,496,528,1024)$, respectively.
The remaining $12$ cases are open.
For five of these, the existence of even a difference set of cardinality $D$ for a group of order $N$ is unresolved, namely when $(D,N)$ is $(190,400)$, $(325,676)$, $(378,784)$, $(385,925)$ and $(280,931)$.
This leaves just seven open cases that should probably bear the most scrutiny.
They have $(R,D,E,N)$ parameters
$(11,12,33,45)$, $(19,20,76,96)$, $(29,30,145,175)$, $(105,126,225,351)$, $(55,56,385,441)$, $(71,72,568,640)$ and $(89,90,801,891)$.
Interestingly, like those of Theorem~\ref{thm.infinite family},  these parameters all match those arising from certain McFarland difference sets with one exception: $(105,126,225,351)$ is instead consistent with a certain \textit{Spence} difference set~\cite{JungnickelPS07}.
Of these, $(19,20,76,96)$, $(105,126,225,351)$ and $(71,72,568,640)$ seem the most promising since ETFs with the same parameters as those guaranteed by Theorem~\ref{thm.ECTFF from PDS} are already known to exist~\cite{FickusM16}.
It would be more surprising if $12$- and $33$-element paired difference sets for either of the two abelian groups of order $45$ existed since this would give an $\ETF(11,33)$, which would be the smallest new ETF discovered in years.
Our numerical work indicates that such paired difference sets do not exist.
This itself raises another open problem: do all paired difference sets consist of $2^{M-1}(2^M-1)$- and $2^{M-1}(2^M+1)$-element subsets of an abelian group of order $2^{2M}$?

\section*{Acknowledgments}
The authors thank Prof.~Dustin~G.~Mixon and the two anonymous reviewers for their thoughtful comments.
We are especially grateful for the anonymous remark that led to Theorem~\ref{thm.necessary}.
This work was partially supported by NSF DMS 1830066, and began during the Summer of Frame Theory (SOFT) 2016.
The views expressed in this article are those of the authors and do not reflect the official policy or position of the United States Air Force, Department of Defense, or the U.S.~Government.

\end{document}